\documentclass[13pt,a4paper, oneside]{amsart}
\usepackage{}
\usepackage{mathrsfs}
\usepackage{amsfonts}
\usepackage{xy}
\usepackage{cite}
\usepackage{amssymb}
\usepackage{color, latexsym}
\usepackage{graphicx}
\usepackage{amsbsy,amscd}
\usepackage{fancyhdr}
\usepackage[colorlinks=true,
           citecolor=blue,
           linkcolor =red]{hyperref}
\xyoption{all}
\allowdisplaybreaks

\fontsize{9pt}{9pt}\selectfont
\def\){\big)}
\def\({\big(}

\def\ga{\alpha}
\def\gb{\beta}
\def\gd{\delta}
\def\gma{\gamma}
\def\gl{\lambda}

\def\bz{\mathbb{Z}}
\def\bc{\mathbb{C}}

\def\fa{\mathfrak{a}}
\def\fb{\mathfrak{b}}
\def\fc{\mathfrak{c}}

\def\fs{\mathfrak{s}}

\def\ft{\mathfrak{t}}
\def\fu{\mathfrak{u}}

\def\cf{\mathcal{F}}

\def\mcr{\mathcal{R}}

\def\sa{\mathscr{A}}
\def\sb{\mathscr{B}}

\def\sh{\mathscr{H}}

\def\sr{\mathscr{R}}

\def\And{\text{\ and\ }}

\def\for{\text{\ for\ }}

\def\If{\text{\ if\ }}
\def\Or{\text{\ or\ }}

\def\otherwise{\text{\ otherwise}}
\def\shape{\text{Shape}}

\def\mod{\mathrm{mod\,}}

\def\gr{\mathrm{gr}}

\def\ker{\mathrm{Ker}}

\def\res{\mathrm{res}}
\def\std{\mathrm{Std}}

\def\mpn{\mathscr{P}(m, n)}
\def\rr{\rightarrow}
\def\la{\langle}
\def\ra{\rangle}

\def\Number#1{\refstepcounter{equation}
              \leqno(\theequation)\if*#1%
              \else\def\@currentlabel{{\rm\theequation}}\label{#1}%
              \fi}

\def\lexp#1#2{\kern\scriptspace\vphantom{#2}^{#1}\kern-\scriptspace#2}

\def\@underbar#1#2{\settowidth{\@tempdimb}{$#1#2$}\@tempdimb=0.8\@tempdimb
                   \ooalign{$#1#2$\crcr%
                         \hfil\rule[-.5mm]{\@tempdimb}{.4pt}\hfil}}
\newdimen\hoogte    \hoogte=14pt    
\newdimen\breedte   \breedte=14pt   
\newdimen\dikte     \dikte=0.5pt    

\newenvironment{Point}[2]%
  {\ifx*#2\let\pointlabel\relax\else\def\pointlabel{#2}\fi
   \refstepcounter{equation}\trivlist
   \item[\hskip\labelsep\theequation.
         \ifx\pointlabel\relax\else\space\pointlabel\space\fi]
   \ignorespaces #1
  }{\relax}

\newenvironment{young}{\begingroup
       \def\vr{\vrule height0.8\hoogte width\dikte depth 0.2\hoogte}
       \def\fbox##1{\vbox{\offinterlineskip
                    \hrule height\dikte
                    \hbox to \breedte{\vr\hfill##1\hfill\vr}
                    \hrule height\dikte}}
       \vbox\bgroup \offinterlineskip \tabskip=-\dikte \lineskip=-\dikte
            \halign\bgroup &\fbox{##\unskip}\unskip  \crcr }
       {\egroup\egroup\endgroup}
\def\diagram#1{\relax\ifmmode\vcenter{\,\begin{young}#1\end{young}\,}\else%
              $\vcenter{\,\begin{young}#1\end{young}\,}$\fi}

\numberwithin{equation}{section}

\swapnumbers
\newtheorem{theorem}[equation]{Theorem}
\newtheorem{lemma}[equation]{Lemma}
\newtheorem{proposition}[equation]{Proposition}
\newtheorem{corollary}[equation]{Corollary}

\theoremstyle{definition}
\newtheorem{definition}[equation]{Definition}
\newtheorem{example}[equation]{Example}

\newtheorem{assume}[equation]{Assumption}

\theoremstyle{remark}
\newtheorem{remark}[equation]{Remark}

\begin{document}
\setlength{\itemsep}{-0.25cm}
\fontsize{12}{\baselineskip}\selectfont
\setlength{\parskip}{0.34\baselineskip}
\vspace*{-12mm}
\title[Cyclotomic Hecke algebras]{\fontsize{13}{\baselineskip}\selectfont Matrix units and Schur elements\\ \vspace{1.5\jot} for the degenerate cyclotomic Hecke algebras}
\author{Deke Zhao}
\address{\begin{tabular}{l}
          \textsl{School of Applied Mathematics, Beijing Normal University at Zhuhai, Zhuhai, 519087, China}\\
         \textsl{Academy of Mathematics and Systems Science, Chinese Academy of Sciences, Beijing 100190, China}\\
           \textsl{E-mail address: deke@amss.ac.cn}
           \end{tabular}}
\thanks{The author is very grateful to Professors Nanhua Xi and Yang Han for their invaluable helps and constant encouragement.
The author would like to thank Dr.~Ming Fang for his generous support and for some helpful conversations.
This research was supported by the National Natural Science Foundation of China (No.~11001253 and No.~11101037).}
\subjclass[2010]{Primary 16G99; Secondary 20C20, 20G05.}



\keywords{(Degenerate) cyclotomic Hecke algebras; Schur elements; Sepcht modules}
\begin{abstract}
The paper uses the cellular basis of the (semi-simple) degenerate cyclotomic Hecke algebras to investigate these algebras exhaustively. As a consequence, we describe explicitly the ``Young's seminormal form'' and a orthogonal bases for Specht modules and determine explicitly the closed formula for the natural bilinear form on Specht modules and Schur elements for the degenerate cyclotomic Hekce algebras.
\end{abstract}
\maketitle
\vspace*{-10mm}
\section{Introduction}
The cyclotomic Hecke algebras or Ariki-Koike algebras were introduced independently by Ariki and Koike \cite{AK94} who were
interested in them because they are a natural generalization of the Iwahori-Hecke algebras of types $A$ and $B$, and by Brou\'{e} and Malle \cite{BM:cyc} who conjectured these algebras should play a role in the modular representations of finite groups of Lie types. The cyclotomic Hecke algebras  also appear in a different guise in the work of Cherednik \cite{Cherednik} as a family of cyclotomic quotient of the (extended) affine Hecke algebras. Now the cyclotomic Hecke algebras paly a fundamental role in many branches of mathematics and were investigated in various settings.

The cyclotomic Hecke algebras have a natural rational degeneration\leavevmode\vrule height 2pt depth -1.6pt width 15pt degenerate cyclotomic Hecke algebras, which is a family of finite dimensional quotient algebras of the degenerate (extend) affine Hecke algebras, see \cite[Chap.\,7]{Kbook}. Now it turns out that the category of integral representations of degenerate affine Hecke algebras consists precisely of all inflations of all its cyclotomic quotients and
the study of degenerate cyclotomic Hecke algebras has some independent interests, see for example \cite{BK-Schur-Weyl,BK-Math.Z.}.
The purpose of this paper is to investigate exhaustively the degenerate cyclotomic Hecke algebras using the cellular basis of these algebras, some investigations are outlined in Ariki, Mathas and Rui's work \cite[\S 6]{AMR}, which should lead to a better understanding of these algebras.

 The paper is largely inspired by Mathas's work \cite{Mathas-J-Algebra}, which originates form Murphy's classical work \cite{murphy:basis,murphy:hecke}. For semi-simple degenerate cyclotomic Hecke algebras, we explicitly construct the matrix unites of it (Theorem~\ref{Them orthogonal basis}). To do this we give a detailed investigation on Specht modules, which enables us to determine explicitly the ``Young's seminormal form" and a orthogonal bases for Specht modules (Theorem~\ref{Prop-s_i f_su} and Corollary~\ref{Cor orthogonal basis of Spechts}). Furthermore, we obtain a closed formula on the natural bilinear forms on Specht modules (Theorem~\ref{gamma properties}), which answer a question of Ariki, Mathas and Rui \cite[\S6.9]{AMR}. As an application, we obtain a complete set of primitive idempotents of semi-simple degenerate cyclotomic Hecke algebras (Theorem~\ref{Them idempotents}), and describe `explicitly'  the Brundan-Kleshchev isomorphism (\cite[Corollary 1.3]{BK-Block}) between degenerate cyclotomic Hecke algebras and cyclotomic Hecke algebras in the semi-simple case.

 It is well-known that the Schur elements play a powerful role in the representation theory of symmetric algebras, see for example \cite[Chap.\,9]{Curtis-Reiner} and \cite[Chap.\,7]{Geck}. In the case of the degenerate cyclotomic Hecke algebras, Brundan and Kleshchev's work [\cite{BK-Schur-Weyl}, Theorem A2] showed that these algebras are symmetric algebras for all parameters, which enable us to use the Schur elements to determine when Specht modules are projective irreducible and whether the algebra is semi-simple. In the paper we determine explicitly Schur elements for the degenerate cyclotomic Hecke algebras by computing the trace form on some nice idempotents of these algebras (Theorem~\ref{Them Schur elements}). In the process of our computation, we introduce firstly the $\bar{\,\,}$ operation for the degenerate cyclotomic Hecke algebras (Definition~\ref{Def bar-operation}), which is very different from that one introduced by Mathas \cite[\S3]{Mathas-J-Algebra} for cyclotomic Hecke algebras, and then by a straightforward computation based on two key Lemmas~\ref{Lem key} and \ref{Lem gamma_t_lambda}.

 The lay-out of this paper as follows. In section~\ref{Sec: def of DCA} we recall the definition of the degenerate cyclotomic Hecke algebras and the non-degenerate trace form on these algebras. In the long but fundamental Section~\ref{Sec: Cellular basis}, use the cellular bases we give a thorough investigation of the (semi-simple) degenerate cyclotomic Hecke algebras, for example the ``Young's seminormal form" and a orthogonal bases for Specht modules and the primitive idempotents of these algebras. The dual Specht module, or equivalently, a dual cellular basis of the degenerate cyclotomic Hecke algebras is considered in Section~\ref{sec: dual Spechts}, which is applied in Section~\ref{Sec: nice primitive idempotents} to give some nice idempotents of these algebras. Section~\ref{Sec: comuputaion} gives a direct computation of the trace form on these nice idempotents without any assumption, which gives a different proof of the non-degeneration of the trace form.  Finally, a closed formula for Schur elements for the degenerate cyclotomic Hecke algebras is given in Section~\ref{Sec: Schur elements}. Throughout this paper, we assume that $R$ is a commutative ring and that $m$ and $n$ are positive integers unless otherwise stated.
\section{Degenerate cyclotomic Hecke algebras}\label{Sec: def of DCA}
In this section, we recall that the definitions of degenerate cyclotomic Hecke algebras and of degenerate affine Hecke algebras. The non-degenerate trace form on the degenerate cyclotomic Hecke algebras and some basic facts are reviewed briefly.

\begin{Point}{}* Let $m, n$ be positive integers. Recall from \cite{shephard-toda} or \cite{cohen} that the complex
  reflection group $W_{m,n}$ of type $G(m, 1, n)$ is the finite group generated by elements $s_0, s_1, \dots, s_{n-1}$ subject to the relations
 $$\begin{aligned}&s_0^m=1, \quad
 s_0s_1s_0s_1=s_1s_0s_1s_0\\
 &s_i^2=1, \quad s_is_{i+1}s_i=s_{i+1}s_{i}s_{i+1}, \quad i>1\\
 & s_is_j=s_js_i, \quad |i-j|>1.\end{aligned}$$
 In particular, the subgroup $\la s_1, \dots, s_{n-1}\ra$ of $W_{m,n}$ is isomorphic to the  symmetric group $S_n$ of order
 $n$ with simple transposition $s_i=(i,i+1)$ for $i=1, \dots, n-1$.
It is well-known that $W_{m,n}\cong(\mathbb{Z}/m\bz)^n\rtimes S_n$. Clearly, $W_{1,n}$ is the Weyl group of type $A_n$ and
 $W_{2,n}$ is the Weyl group of type $B_n$.
 \end{Point}

\begin{definition}\label{Def DCHA}Let $R$ be a commutative ring and $Q=\{q_1, \dots, q_m\}\subset R$.
The {\it degenerate cyclotomic Hecke algebra} is the unital
associative $R$-algebra  $\sh\!:=\sh_{m,n}(Q)$ generated by
$s_0,s_1,\dots,s_{n-1}$ and subjected to relations

\begin{enumerate}\item
$(s_0-q_1)\dots(s_0-q_m)=0$,
\item $s_0(s_1s_0s_1+s_1)=(s_1s_0s_1+s_1)s_0$,
\item $s_i^2=1$,                \quad $1\le i<n$,
\item $s_is_{i+1}s_i=s_{i+1}s_is_{i+1}$, \quad  $1\le i<n-1$,
\item $s_is_j=s_js_i$,                   \quad$|i-j|>1$.
\end{enumerate}
The elements $x_1\!:=s_0$
and $x_{i+1}\!:=s_ix_is_i+s_i$ for $i=1, \dots, n-1$ of $\sh$ are called the \textit{Jucys-Murphy elements} of $\sh$.
\end{definition}
Clearly, $\sh_{1,n}(Q)$ is exactly the group algebra $RS_n$ and $x_1, \dots, x_n$ are
algebraically dependent, moreover, for all $i$, the minimal polynomial of $x_i$ can be determined explicitly, see Corollary~\ref{Cor x_k F_k}(ii).  The following facts will be used frequently.

 \begin{lemma}\label{Lem si xj}Suppose that $1\le i<n$ and $1\le j, k\le n$. Then
\begin{enumerate}
\item  $s_jx_j-x_{j+1}s_j=-1$ and $s_{j-1}x_j-x_{j-1}s_{j-1}=1$.
 \item $s_ix_j=x_js_i$ if $i\neq j-1, j$.
 \item $x_jx_k=x_kx_j$ if $1\le j, k\le n$.
 \item $s_j(x_jx_{j+1})=x_jx_{j+1}s_j$ and $s_j(x_j+x_{j+1})=(x_{j}+x_{j+1})s_j$.
 \item if $a\in R$ and $i\neq j$ then $s_i$ commutes with $(x_1-a)(x_2-a)\cdots(x_j-a)$.
 \end{enumerate}
  \end{lemma}
\begin{proof}(i) follows directly by the definition of $x_{j+1}$ and \ref{Def DCHA}(iii).

(ii) If $j=1$ then $s_ix_1=s_is_0=x_1s_i$ whenever $i\neq 1$ according to \ref{Def DCHA}(v). Assume that for all
$j=l\geq1$, $s_ix_l=x_ls_i$ if $i\neq l-1, l$. Suppose that $j=l+1$. Then, by induction and \ref{Def DCHA}(v),
 $s_ix_j=s_i(s_l+s_lx_ls_l)=x_{j}s_i$ if
$i\neq l-1,l,l+1$. On the other hand,
 $$\begin{aligned}s_{l-1}x_{l+1}-x_{l+1}s_{l-1}&=s_{l-1}s_{l}-
s_{l}s_{l-1}+s_{l-1}s_lx_ls_l-s_{l}x_ls_ls_{l-1}\\
&=s_{l-1}s_{l}s_{l-1}x_{l-1}s_{l-1}s_l-s_ls_{l-1}x_{l-1}s_{l-1}s_ls_{l-1}\\
&=s_ls_{l-1}(s_lx_{l-1}-x_{l-1}s_{l})s_{l-1}s_l,
\end{aligned}$$
so, by induction, $s_{l-1}x_{l+1}=x_{l+1}s_{l-1}$ and (ii) is proved.

(iii) Without loss of generality, we may assume that $1\le j\le k\le n$. Then \ref{Def DCHA}(ii) and \ref{Def DCHA}(v) imply that $x_1x_k=x_kx_1$ for $k=1, \dots, n$. Assume
that for all $j=l\geq 1$ and for all $l\le k\le n$, $x_ix_k=x_kx_i$. Suppose that $j=l+1$.
Then, by (ii) and induction,  $x_{j}x_k=(s_{l}+
s_{l}x_{l}s_{l})x_k=x_kx_i$ for all $k\geq i=l+1$. Hence $x_ix_k=x_kx_j$ for all $1\le j,k\le n$.

(iv) The first equality follows by \ref{Def DCHA}(iii) and (iii), and
the second one follows by (ii).

(v) Let $X_j(a)\!:=(x_1-a)\dots(x_j-a)$. Obviously $s_iX_1(a)=X_1(a)s_i$ if $i\neq 1$.
 Assume that for all $j=l\geq 1$, $s_iX_j(a)=X_j(a)s_i$ if $i\neq j$. Then, by induction and (ii),
 $$s_iX_{l+1}(a)=s_iX_{l}(a)(x_{l+1}-a)=X_{l+1}(a)s_i$$
  if $i\neq l, l+1$. On the other hand, by induction and (iv),
$$s_lX_{l+1}(a)=X_{l-1}(a)s_l(x_l-a)(x_{l+1}-a)=X_{l-1}(a)(s_lx_lx_{l+1}-as_{l}(x_l+x_{l+1})
-a^2s_l)=X_{l+1}(a)s_l.$$ Thus $s_iX_{l+1}(a)=X_{l+1}(a)s_i$ for $i\neq l+1$. The proof is completed.
\end{proof}

  Recall from \cite{Drinfeld} or \cite{Lusztig89} that the \textit{degenerate affine Hecke algebra}
   $\widehat{\sh}^{\text{aff}}_n$ of $\mathrm{GL}_n(\bc)$ is the associated $R$-algebra which is
    equal as an $R$-module to the tensor product $R[y_1, \dots, y_n]\otimes_RRS_n$ of the polynomial
     algebra $R[y_1, \dots, y_n]$ and the group algebra $RS_n$. Multiplication is defined so that
     $R[y_1, \dots, y_n]$ and $RS_n$ are subalgebras, and in addition
   $$\begin{aligned}&s_iy_{i+1}-y_is_i=1\\&s_iy_j=y_js_i \If i\neq j-1, j.\end{aligned}$$

The following lemma shows that the degenerate cyclotomic Hecke algebras are the cyclotomic quotients of the degenerate affine Hecke algebra, which play a fundamental role in the study of the degenerate affine Hecke algebra, see for example \cite[Chap.~7]{Kbook}.
\begin{lemma} Let $J_Q$ denote the two-sided ideal of $\widehat{\sh}^{\text{aff}}_n$ generated
 by $(y_1-q_1)\cdots(y_1-q_m)$. Then there is a surjective homomorphism of algebras
 $\pi: \widehat{\sh}^{\text{aff}}_n\rr\sh$ such that $y_i\mapsto x_i$ for each $i$ and
  $s_j\mapsto s_j$ for each $j$. Then $\pi$ is a surjective homomorphism and
  $\ker\pi=J_{Q}$.
\end{lemma}
\begin{proof}The Lemma follows by definitions and Lemma~\ref{Lem si xj}. \end{proof}

For  degenerate cyclotomic Hecke algebra $\sh$, we have the following important theorem.

\begin{theorem}[\cite{Kbook}, Theorem 7.5.6]The degenerate cyclotomic Hecke algebra $\sh$ is a free $R$-module with basis
$\{x_1^{i_1}x_2^{i_2}\cdots x_n^{i_n}w\mid0\le i_1, \dots, i_n<m, w\in S_n\}$.\end{theorem}

Now we recall the non-degenerate trace form on $\sh$ constructed by Brundan and Kleshchev \cite{BK-Schur-Weyl},
which is our main investigation in this paper. To describe this one and for the computation in Section~\ref{Sec: comuputaion}, we introduce some notation.
Let $R_m[x_1, \dots, x_n]$ be the level $m$ truncated polynomial algebra, that is,
 the quotient of the polynomial algebra $R[y_1, \dots, y_n]$ by the two-sided ideal
  generated by $y_1^m, \dots, y_n^m$.
Also define a grading on the twist tensor algebra $R_m[y_1, \dots, y_n]
\rtimes\!\!\!\!\!\!\!\bigcirc RS_n$ by declaring that each $y_i$ is
of degree 1  and each $w\in S_n$ is of degree 0.

\begin{lemma}[\cite{BK-Schur-Weyl}, Lemma A1]Let $\hat{\tau}:R_m[y_1, \dots, y_n]
\rtimes\!\!\!\!\!\!\!\bigcirc RS_n\rr R$ be the $R$-linear map defined by
$$\hat{\tau}(y_1^{i_1}\cdots y_n^{i_n} w)
        :=\left\{\begin{array}{ll} 1,&\text{ if }i_1=\cdots=i_n=m-1 \text{ and } w=1,\\
                0,&\text{ otherwise}.
                \end{array}\right.$$
Then $\hat{\tau}$ is a non-degenerate trace form on $R_m[y_1, \dots, y_n]\!
\rtimes\!\!\!\!\!\!\!\bigcirc RS_n$.
\end{lemma}

Now let $l=(m-1)n$ and define a filtration $$RS_n=\cf_0\sh\subseteq\cf_1\sh\subseteq
\cdots\subseteq \cf_l\sh=\sh$$
of $\sh$ by defining that $$\cf_r\sh:=\text{Span}_{R}\{x_1^{i_1}x_2^{i_2}\cdots x_n^{i_n}w|
i_1+\cdots+i_n\le r, w\in S_n\}.$$
For any $0\le i\le l$, let $$\gr_i: \sh\rr \cf_i\sh/\cf_{i-1}\sh$$
be the map sending an element of $\sh$ to its degree $i$ graded component.

\begin{lemma}[\cite{BK-Schur-Weyl}, Lemma 3.5]\label{Lem graded iso} There is a well defined isomorphism of
 graded algebras $$\psi_n:R_m[y_1, \dots, y_n]\rtimes\!\!\!\!\!\!\!\bigcirc RS_n\rr \gr \sh$$
such that $y_i\mapsto \gr_1x_i$ for each $i$ and $s_j\mapsto \gr_0s_j$ for each $j$.\end{lemma}

The following surprising theorem says that $\sh$ is a symmetric algebra for all parameters $q_1, \dots, q_m$ in $R$,
Theorem~\ref{Them tau(Zw)} gives a different proof on the non-degeneration of the trace form $\tau$ on $\sh$.
\begin{theorem}[\cite{BK-Schur-Weyl}, Theorem A2]Let $\tau:\sh\rr R$ be the $R$-linear map determined by
$$\tau(x_1^{i_1}\cdots x_n^{i_n} w)
        :=\left\{\begin{array}{ll} 1,&\text{ if }i_1=\cdots=i_n=m-1 \text{ and } w=1,\\
                0,&\text{ otherwise}.
                \end{array}\right.$$
Then $\tau=\hat{\tau}\circ \gr_l$, which is a non-degenerate trace form on $\sh$.
\label{Thm-def tau}\end{theorem}

\begin{Point}{}*
 Note that $\tau$ is essentially independent of the choice of basis of $\sh$. We will need the following easy verified facts.
\noindent\begin{enumerate}
\item Suppose that $h_1,h_2\in\sh$. Then $\tau(h_1h_2)=(h_2h_1)$.
\item Suppose that $w,v\in S_n$ and that $0\le i_1, \dots, i_n<m$. Then
$$\tau(x_1^{i_1}x_2^{i_2}\ldots x_n^{i_n}wv)
      =\begin{cases} 1,&\text{if } i_1=\dots=i_n=m-1\text{ and } w=v^{-1},\\
                        0,&\text{otherwise}.
\end{cases}$$
\end{enumerate}\label{tau properties}\end{Point}

\begin{assume}\label{ppoly}In this paper we will mainly be concerned with the semi-simple degenerate cyclotomic Hecke
algebras; these were classified by Ariki-Mathas-Rui~\cite[Theorem 6.11]{AMR} who showed
that when $R$ is a field $\sh$ is semisimple if and only if
$$P_\sh(Q)=n!
  \prod_{1\le i<j\le m}\prod_{|d|<n}(d+q_i-q_j)\neq 0.$$
For most of what we do it will be
enough to assume that $R$ is a ring in which $P_\sh(Q)$ is invertible.\end{assume}

 \section{A orthogonal basis of the degenerate cyclotomic Hecke algebra}\label{Sec: Cellular basis}
In this section, we review the cellular basis and Specht modules of the degenerate cyclotomic Hecke algebra $\sh$. Using the cellular basis,  we first give the ``Young's seminorma form" for Specht modules (Theorem~\ref{Prop-s_i f_su}), which enables us to give a orthogonal bases of Specht modules (Corollary~\ref{Cor orthogonal basis of Spechts}) and a orthogonal bases and primitive idempotents of $\sh$ (Theorems~\ref{Them orthogonal basis} and \ref{Them idempotents}). Furthermore,  we obtain a closed formula for the natural bilinear form on Specht modules, which answer a question of Ariki, Mathas and Rui \cite[\S6.9]{AMR}, and give a differential proof of the Brundan-Kleshchev isomorphism between the cyclotomic Hecke algebras and the degenerate ones in \cite[Corollary 1.3]{BK-Block} when these algebras are semi-simple.

\begin{definition}[Graham and Lehrer~\cite{GL}]\label{Def-Cellular-algebras}
    Let $A$ an $R$-algebra.
    Fix a partially ordered set $\Lambda=(\Lambda,\unrhd)$ and for each
    $\lambda\in\Lambda$ let $T(\lambda)$ be a finite set. Finally,
    fix $C^\lambda_{\fs\ft}\in A$ for all
    $\lambda\in\Lambda$ and $\fs,\ft\in T(\lambda)$.
    Then the triple $(\Lambda,T,C)$ is a \textit{cell datum} for $A$ if:

    \begin{enumerate}
    \item $\{C^\lambda_{\fs\ft}|\lambda\in\Lambda\text{ and }\fs,\ft\in
        T(\lambda)\}$ is an $R$-basis for $A$;
    \item the $R$-linear map $*: A\rr A$ determined by
        $(C^\lambda_{\fs\ft})^*=C^\lambda_{\ft\fs}$, for all
        $\lambda\in\Lambda$, $\fs,\ft\in T(\lambda)$ is an
        anti-automorphism of $A$;
    \item for all $\lambda\in\Lambda$, $\fs\in T(\lambda)$ and $a\in A$
        there exist scalars $r_{\fs\fu}(a)\in R$ such that
        \vspace{-1.6\jot}$$aC^\lambda_{\fs\ft}
            =\sum_{\fu\in T(\lambda)}r_{\fs\fu}(a)C^\lambda_{\fu\ft}
                     \pmod{A^{\rhd\lambda}},$$\vspace{-1.6\jot}
            where
    $A^{\rhd\lambda}=\text{Span}_{R}%
      \{C^\mu_{\fa\fb}|\mu\rhd\lambda\text{ and }\fa,\fb\in T(\mu)\}$.
    \end{enumerate}
    \noindent An algebra $A$ is a \textit{cellular algebra} if it has
    a cell datum and in this case we call
    $\{C^\lambda_{\fs\ft}|\fs,\ft\in T(\lambda), \lambda\in\Lambda\}$
    a \textit{cellular basis} of $A$.

\end{definition}

\begin{Point}{}* Recall that an {\it $m$-multipartition} of $n$ is a ordered $m$-tuple
$\lambda=(\gl^1; \dots; \gl^m)$ of partitions $\lambda^{i}$ such that
$n=\sum_{i=1}^m|\lambda^{i}|$. Denote by $\mpn$ the set of all $m$-multipartitions of $n$,
 which is a poset under dominance $\unrhd$, where $\gl\unrhd\mu$, if
$$
\displaystyle\sum_{k=1}^{i-1}|\gl^{k}|+\sum_{l=1}^j\gl_l^{i}\geq
\sum_{k=1}^{i-1}|\mu^{k}|+\sum_{l=1}^j\mu_l^{i},
$$
for all $1\leq i\leq m$ and $j\geq 1$. We write $\gl\rhd\mu$ if $\gl\unrhd\mu$ and $\gl\neq \mu$.

Suppose that $\gl$ is an $m$-multipartition of $n$ and let $\mathbf{m}=\{1, \dots, m\}$.
The {\it diagram} of $\lambda$ is the set
$$ [\gl]:=\{(i,j,c)\in\bz_{>0}\times\bz_{>0}\times \mathbf{m}|1\le j\le\lambda^c_i\}.$$
The elements of $[\gl]$ are the nodes of $\gl$; more generally, a node is any element of
$\bz_{>0}\times \bz_{>0}\times \mathbf{m}$. We  may and will identify $[\lambda]$ with
 the $m$-tuple of diagrams of the partitions $\lambda^{c}$, for $1\le c\le m$. Recall that a node
$y\notin[\lambda]$ is an {\it addable node} for $\lambda$ if
$[\lambda]\cup\{y\}$ is the diagram of an $m$-multipartition, and denote by $\sa(\gl)$ the set of
all addable nodes for $\gl$; similarly,
$y\in[\lambda]$ is a {\it removable node} for $\lambda$ if
$[\lambda]\setminus\{y\}$ is the diagram of an $m$-multipartition,  and denote by
$\mathscr{R}(\gl)$ the set of all removable nodes for $\gl$.

 A {\it $\lambda$-tableau} is a bijection $\ft: [\lambda]\rightarrow\{1,2,\dots,n\}$, and if
$\ft$ is a $\lambda$-tableau write $\text{Shape}(\ft)=\lambda$. As with diagrams,
we may and will think of a tableau $\ft$ as an $m$-tuple of tableaux
$\ft=(\ft^1; \dots; \ft^m)$, where $\ft^{c}$ is a $\lambda^{c}$-tableau, for $1\le c\le m$.  The
tableaux $\ft^{c}$ are called the {\it components} of $\ft$.
A tableau is {\it standard} if in each component the entries increase along the rows and
down the columns; let $\std(\lambda)$ be the set of standard
$\lambda$-tableaux.

Given an $m$-multipartition $\lambda$ let $\ft^{\gl}$ be the
$\lambda$-tableau with the numbers $1,2,\dots,n$ entered in
order first along the rows of $\ft^{\lambda^{1}}$ and then the rows
of $\ft^{\lambda^{2}}$ and so on.
The symmetric group $S_n$ acts from the right on the set of
$\lambda$-tableaux; let
$S_\lambda
     =S_{\lambda^{1}}\times\dots\times S_{\lambda^{m}}$
be the row stabilizer of $\ft^{\gl}$. For any $\lambda$-tableau $\ft$
let $d(\ft)$ be the unique element of $S_n$ such that
$\ft=\ft^{\gl} d(\ft)$ and denote by $\ell(\ft)$ the length of $d(\ft)$. \end{Point}

\begin{example} Let $\gl\!\!=\!\!(3\cdot2;2\cdot1;1)$ be a $3$-multipartition of $9$. Then $S_{\gl}\!\!=\!\!S_{\{1,2,3\}}\!\!\times\!\! S_{\{4, 5\}}\!\!\times\!\! S_{\{6, 7\}}\!\!\times\!\! S_{\{8\}}\!\!\times\!\! S_{\{9\}}$,
$$\begin{array}{ccc}\protect{
[\gl]\!\!=\!\!\left(\begin{array}{ccc}\diagram{&&\cr &\cr };\, &\diagram{&\cr \cr};\, &\diagram{\cr}\end{array}\right)}
\,&\protect{
\ft^{\gl}\!\!=\!\!\left(\begin{array}{ccc}\diagram{1&2&3\cr 4 &5\cr };\, &\diagram{6&7\cr 8\cr};\, &\diagram{9\cr}\end{array}\right)}\,&
\protect{\ft\!\!=\!\!\left(\begin{array}{ccc}
\diagram{2&4&6\cr 3&5\cr };\, &\diagram{7&8\cr 9\cr};\, &\diagram{1\cr}
\end{array}\right)}\\
\protect{\mathscr{R}(\gl)\!\!=\!\!\left(\begin{array}{ccc}
\diagram{&&--\cr &--\cr };\, &\diagram{&--\cr --\cr}; &\diagram{--\cr}
\end{array}\right)}&\, \protect{\mathscr{A}(\gl)\!\!=\!\!\left(\begin{array}{ccc}
\diagram{&&&+\cr &&+\cr +\cr};\, &\diagram{&&+\cr &+\cr +\cr};\, &\diagram{&+\cr +\cr}
\end{array}\right)}  & d(\ft)\!\!=\!\!(1,2,4,3,6,7,8,9),
\end{array}
$$
where $\diagram{--\cr }$ (resp., $\diagram{+\cr }$) means the node is removable (resp., addable) and $S_{\{1, 2,3\}}$ is the symmetric group on letters $1,2,3$ and so on.
\end{example}

Let $<$ be the Bruhat order on $S_n$. The following fact was first proved by Ehresmann and then rediscovered
by Dipper and James \cite{DJ:reps}, see for example, \cite[Theorem 3.8]{Mathas-book}.

\begin{lemma}[Ehresmann Theorem] Suppose that $\fs$ and $\ft$ are standard
tableaux of the same shape. Then $d(\fs)<d(\ft)$ if and only if $\fs\rhd\ft$.\label{Lem Ehresmann}\end{lemma}

\begin{Point}{}* Let $*$ be the $R$-linear anti-automorphism of $\sh$ determined by
$x_i^*=x_i$ for all $1\le i\le n$, $s_j^*=s_j$ for all $1\le j\le n-1$. Then $w^*=w^{-1}$
 for all $w\in S_n$ and $f^*=f$ for all $f\in R[x_1, \dots, x_n]\subset \sh$.
\label{*-anti-auto}\end{Point}

\begin{definition}\label{Def x_lamda u^+_lamda}Suppose that  $\gl=(\gl^1; \dots; \gl^m)$ is an $m$-multipartition of $n$
 and define $a_i=\sum_{j=1}^{i-1}|\gl^j|$ for $1\le i\le m$ with $a_1=0$. Let
 $m_\lambda=x_\lambda u_\gl^+$, where
\vspace{-1.6\jot}
$$ x_\lambda:=\sum_{w\in S_\lambda}w\qquad\text{and}\qquad
u_\lambda^+:=\prod_{i=2}^m\prod_{k=1}^{a_i}(x_k-q_i).$$
\vspace{-1.6\jot}
Finally, given standard $\lambda$-tableaux $\fs$ and
$\ft$ let $m_{\fs\ft}=d(\fs)^*m_\lambda d(\ft)$.
\end{definition}

It follows from Lemma~\ref{Lem si xj}(v) that all of elements in $RS_{\gl}$ commute
with $u_{\gl}^+$, in particular,  $m_{\gl}=x_\lambda u_\lambda^+=u_{\gl}^+x_{\gl}$.
Observe that $m_{\gl}=m_{\ft^{\gl}\ft^{\gl}}$ and $m_{\fs\ft}^*=m_{\ft\fs}$
for standard $\gl$-tableaux $\fs$ and $\ft$.
Whenever we write $m_{\fs\ft}$ in what follows $\fs$ and $\ft$ will be
standard tableaux of the {\it same shape} (and similarly, for
$f_{\fs\ft}$ etc.).

\begin{theorem}[\cite{AMR}, THEOREM 6.3]
The degenerate cyclotomic Hecke algebra $\sh$ is free as an $R$-module with
cellular basis $\{m_{\fs\ft}|\fs,\ft\in\std(\lambda) \for  \lambda \text{ an } m\text{-multipartition of }n\}$.
\label{std basis}\end{theorem}

 If $\ft$ is any
tableau and $k\geq0$ is an integer, let $\ft\!\downarrow\!k$ be the
subtableau of $\ft$ which contains the integers $1,\dots,k$.
Observe that $\ft$ is standard if and only if $\shape(\ft\!\downarrow\!k)$ is
an $m$-multipartition for all $k=1, \dots, n$. We extend the
dominance order $\unrhd$ on the set of $m$-multipartitions to the set of standard tableaux by defining
$\fs\unrhd\ft$ if $\shape(\fs\!\downarrow\!k)\unrhd\shape(\ft\!\downarrow\!k)$ for all
$k=1,\dots,n$; and write $\fs\rhd\ft$ if $\fs\unrhd\ft$ and
$\fs\neq\ft$. Define the {\it residue} of $k$ in $\ft$ to be $\res_\ft(k)=j-i+q_c$ if $k$
appears in node $(i, j,c)\in \ft$.

\begin{lemma}[cf.~\cite{Mathas-book}, Lemma 3.34] Assume that $R$ is a field and that
 Assumption~\ref{ppoly} holds.  Suppose that $\gl$ and $\mu$ are $m$-multipartition of $n$ and
  let $\fs\in\std(\gl)$ an $\ft\in\std(\mu)$.
\vspace{-0.2cm}
\begin{enumerate}
\setlength\itemsep{1pt}
\item $\fs=\ft$ (and $\gl=\mu$) if and only if $\res_\fs(k)=\res_\ft(k)$ for
$k=1,\dots,n$.
\item Suppose that $\gl=\mu$ and there exists an $i$ such that $\res_{\fs}(k)=
\res_{\ft}(k)$ for all $k\neq i, i+1$. Then either $\fs=\ft$ or  $\fs=\ft(i,i+1)$.
\end{enumerate}\label{s=t or s=t(i i+1)}\label{Lem res s t} \end{lemma}

\vspace{-5\jot}\begin{proof}
Note that the nodes of $\fs$  and $\ft$ containing entry $n$ are removable nodes,
which have distinct residues if they are distinct nodes. Then both parts of Lemma follow by
using the induction argument on $n$. \end{proof}

The Lemma says that when Assumption~\ref{ppoly} holds the residues separate the standard tableaux;
this enables us to give the orthogonal basis of degenerate cyclotomic Hecke algebras
and Specht modules. Notice that both parts of the Lemma can be fail if the assumption does not hold.

If $\gl$ is an $m$-multipartition then  let $\sh^{\rhd\gl}$ be the free $R$-submodule of $\sh$ with
basis $\{m_{\fs\ft}|\fs,\ft\in\std(\mu) \for \mu\rhd \gl\}$.
It follows from Theorem~\ref{std basis} and \ref{Def-Cellular-algebras}(iii)
that $\sh^{\rhd\gl}$ is a two-sided ideal of $\sh$.

\begin{definition}\label{Def Specht} The {\it Specht module} $S^\lambda$ is the left
$\sh$-module $\sh m_\lambda/(\sh m_\lambda\cap\sh^{\rhd\gl})$, which is a submodule of
$\sh/\sh^{\rhd\gl}$.\end{definition}

 Theorem~\ref{std basis} implies that $S^\lambda$ is a free $R$-module with basis
$\{m_\ft|\ft\in\std(\lambda)\}$, where $m_\ft=m_{\ft\ft^\gl}+\sh^{\rhd\gl}$.
Further, by the general theory of cellular algebras, there is a natural associative
bilinear form $\la \,,  \ra$ on
$S^\lambda$ which is determined by either$$\la m_\fs,m_\ft\ra m_\lambda\equiv
m_{\ft^\gl\fs}m_{\ft\ft^\gl}\mod\sh^{\rhd\gl}, \, \Or \,\la m_\fs,m_\ft
\ra m_{\fa\fb}\equiv
m_{\fa\fs}m_{\ft\fb}\mod\sh^{\rhd\gl}, \, \for \fa, \fb,\fs,\ft\in\std(\gl).$$

Now we begin to determine the ``Young seminorm form" for Specht modules. Our start point is the following fact.

\begin{lemma}[cf.\cite{Mathas-book}, Lemma~3.29] Let $\gl$ be an $m$-multipartition of $n$ and let $\fs$ and
$\ft$ be standard $\gl$-tableaux such that $\fs\rhd\ft=\fs(i, i+1)$ for some $i$ with
$1\le i<n$. For $k=1, \dots, n$, suppose that there exist elements $r_{\fs}(k)\in R$ such that
$$x_km_{\fs}=r_\fs(k)m_{\fs}
       +\sum_{\fa\rhd\fs}r_{\fa} m_{\fa} \text{ for some } r_{\fa}\in R.$$
       Then for $k=1, \dots, n$, there exist elements $r_{\fb}\in R$ such that
   $$x_km_{\ft}=r_\ft(k)m_{\ft}
       +\sum_{\fb\rhd\ft}r_{\fb} m_{\fb},$$
       where $r_{\ft}(k)=r_{\fs}(k)$ if $k\neq i, i+1$, $r_{\ft}(i)=r_{\fs}(i+1)$  and $r_{\ft}(i+1)=r_{\fs}(i)$.
\label{Lem-x_km_st}\end{lemma}

\begin{proof}
First note that $m_\ft=s_i m_{\fs}$ because $\fs\rhd\ft$ which implies that $\ell(\fs)<\ell(\ft)$ by Lemma~\ref{Lem Ehresmann}.  Therefore, by Lemma~\ref{Lem si xj}(ii), if $k\neq i, i+1$ then
 $$x_km_{\ft}=x_ks_im_{\fs}=s_ix_km_{\fs}=r_{\fs}(k)m_{\ft}+\sum_{\fa\rhd\ft}r_{\fa} m_{\fa}.$$
 Next suppose that $k=i+1$. Then using the first equality of Lemma~\ref{Lem si xj}(ii)
 $$x_{i+1}m_{\ft}=x_{i+1}s_im_{\fs}=(1+s_ix_i)m_{\fs}=r_{\fs}(i)m_{\ft}+\sum_{\fb\rhd\ft}r_{\fb} m_{\fb}.$$
 The case $k=i$ is similar to the case $k=i+1$ by using the second equality of Lemma~\ref{Lem si xj}(ii).
 Finally, equating the coefficients of these formulae, we complete the proof.
\end{proof}
 \begin{theorem}
[\cite{AMR}, LEMMA 6.6] Suppose that $\gl$ is an $m$-multipartition of $n$ and that
$\fs$ and $\ft$ are standard $\lambda$-tableaux. Suppose
that $k$ is an integer with $1\le k\le n$. Then there exist
$r_{\fa}\in R$ such that
$$\mathsf{(i)}\quad\quad\quad  x_km_{\fs\ft}=\res_\fs(k)m_{\fs\ft}
       +\sum_{\fa\rhd\fs}r_{\fa} m_{\fa\ft}
           \bmod\sh^{\rhd\gl},$$
                   or equivalently, there exist
$r_{\fa}\in R$ such that $$\mathsf{(ii)}\quad\quad\quad x_km_{\fs}=\res_\fs(k)m_{\fs}
       +\sum_{\fa\rhd\fs}r_{\fa} m_{\fa}
           \bmod\sh^{\rhd\gl}.$$
\label{x_k action}\end{theorem}
\begin{proof}
Note that $m_{\fs\ft}d(\ft)^*=m_{\fs\ft^{\gl}}$ for any standard tableaux $\fs$
and $\ft$, we only need to show (ii). First consider the case where $\fs=\ft^{\gl}$,
 then $m_{\fs}=m_{\gl}+\sh^{\rhd\gl}$. Suppose that $k$ appears in node $(i, j,c)\in\fs$
  and that $l$ is the smallest integer appears in component $\ft^{\gl^c}$. Then $l\leq k$.
   Working modulo $\sh^{\rhd\gl}$ and using Lemma~\ref{Lem si xj} (ii), (iv), we obtain that
$$\begin{aligned}
x_km_{\ft^{\gl}}\equiv x_km_{\gl}&=s_{k-1}m_{\gl}+s_{k-1}x_{k-1}s_{k-1}m_{\gl}\\
&=s_{k-1}m_{\gl}+s_{k-1}(s_{k-2}+\cdots+s_{k-2}\cdots s_{l}\cdots s_{k-2})s_{k-1}m_{\gl}
+s_{k-1}\cdots s_{l-1}x_{l-1}s_{l-1}\cdots s_{k-1}m_{\gl}\\
&=(s_{k-1}+q_c)m_{\gl}+s_{k-1}(s_{k-2}+\cdots+s_{k-2}\cdots s_{l-1}\cdots s_{k-2})s_{k-1}m_{\gl}+h,
\end{aligned}$$
where $h=s_{k-1}\cdots s_{l-1}(x_{l-1}-q_c)s_{l-1}\cdots s_{k-1}m_{\gl}=
s_{k-1}\cdots s_{l-1}(x_{l-1}-q_c)u^+_{\gl}s_{l-1}\cdots s_{k-1}x_{\gl}\in\sh^{\rhd\gl}$. Therefore
$$x_km_{\ft^{\gl}}=(s_{k-1}+q_c)m_{\gl}+s_{k-1}(s_{k-2}+\cdots+s_{k-2}\cdots s_{l-1}\cdots
s_{k-2})s_{k-1}m_{\gl} \bmod \sh^{\rhd\gl}.$$
So it is sufficient to show that, for $k=1, \dots, n$,  $$(*) \quad\quad s_{k-1}m_{\gl}+s_{k-1}
(s_{k-2}+\cdots+s_{k-2}\cdots s_{l-1}\cdots
s_{k-2})s_{k-1}m_{\gl}=(\res_{t^{\gl}}(k)-q_c)m_{\gl}\equiv(j-i)m_{\gl}.$$
When $n=1$ there is nothing to prove so by induction we assume that $(*)$ holds for all smaller integers of $n$.
We now proceed by induction on $k$. The case $k=1$ being trivial because
$(x_1-q_c)m_{\gl}\equiv0$ and $\res_{\ft^{\gl}}(k)=q_c$.
Suppose first that  $\gl_{i+1}^c\neq 0$. Let $\mu^c=(\gl_1^c, \dots, \gl^c_{i})$ and
$\mu=(\gl^1; \cdots; \gl^{c-1}; \mu^c; \gl^{c+1}; \dots; \gl^m)$. Then
$a=|\mu|<n$ and $m_{\gl}=m_{\mu}h+\sh^{\rhd\gl}$ for some element $h$ of the subalgebra of $RS_{\gl}$ generate by
$s_{k+1}, \dots, s_{a-1}$. Furthermore, $x_km_{\mu}\in\sh_{m, a}(Q)$ and $h, x_k$
commute by Lemma~\ref{Lem si xj}(ii). Hence working modulo $\sh^{\rhd\gl}$ and arguing by induction on $n$,
$$x_km_{\ft^{\gl}}\equiv x_km_{\mu}h=\res_{\ft^{\mu}}(k)m_{\mu}h\equiv \res_{\ft^{\gl}}(k)m_{\gl}.$$
Thus we may assume that $k$ is in the last row of the $c$-component of $\ft^{\gl}$.

Next suppose that $k$ is not in the first column of $\ft^{\gl^{c}}$ and that we have known the result for small $k$.
Then $k>1$ and $s_{k-1}$ is an element of $S_{\gl^c}$, so $s_{k-1}m_{\gl}=m_{\gl}$ and therefore, by induction on $k$,
$$x_km_{\gl}=(s_{k-1}+s_{k-1}x_{k-1}s_{k-1})m_{\gl}=(1+\res_{\ft^{\gl}}(k-1))m_{\gl}.$$
This reduces us to the case that $k$ appears in the last row and the first column of $\ft^{\gl^c}$. Thus we can assume that
$k=a_{c+1}=\sum_{l=1}^c|\gl^l|$ and that $\gl^c=(\gl^c_1, \dots, \gl_{i-1}^c, 1)$. Let $p$ be the integer
appears in the node $(i-1, 1, c)$ and let $\fu$ be the standard $\gl$-tableau $\ft^{\gl}s_{a_{c+1}-1}\dots s_{p+1}$.
Then the  second last row of $\fu$ contains the integers $p, p+2, \dots, a_{c+1}$ and the last row of $\fu$ contains the
single integer $p+1$.

So far we have shown that $x_km_{\gl}=\res_{\ft^{\gl}}=\res_{\ft^{\gl}}(k)m_{\gl}$ for $a_{c} \le k<a_{c+1}$. We also know
that $(x_{a_c+1}+\dots+x_{a_{c+1}})m_{\gl}=rm_{\gl}$ for some $r\in R$ since $x_{a_c+1}+\dots+x_{a_{c+1}}$ belongs to the
center of $\sh$. Consequently, $x_{a_{c+1}}m_{\gl}=r_{c}m_{\gl}$ for some $r_{c}\in R$. Therefore,
by Lemma~\ref{Lem-x_km_st}
$$x_km_{\fu}=\res_{\fu}(k)m_{\fu}+\sum_{\fa\rhd\fu}r_{\fa}m_{\fa} \quad\text{ and }\quad
x_{k+1}m_{\fs}=r_cm_{\fu}+\sum_{\fb\rhd\fs}r_{\fb}m_{\fb}.$$
Now, by induction, we have
$$\begin{aligned}x_{p+1}m_{\fs}=(1+s_{p}x_p)s_{p}m_{\fs}&=-(1+s_px_p)m_{\fs}+\sum_{\fa\rhd\fs}r_{\fa}m_{\fa}\\
&=(\res_{\ft^{\gl}}(p)-1)m_{\fs}+\sum_{\fa\rhd\fs}r'_{\fa}m_{\fa}\\
&=\res_{\ft^{\gl}}(a_{c+1})m_{\fs}+\sum_{\fa\rhd\fs}r'_{\fa}m_{\fa},\end{aligned}$$
Consequently, $r_{c}=\res_{\ft^{\gl}}(a_{c+1})$ as required.
The general case follows by Lemma~\ref{Lem-x_km_st} and induction argument on $\ell(\fs)$.
\end{proof}

\begin{remark}\label{Remark x_k m_st} (i) If $\fu$ is a standard tableau and $\res_{\fu}(k)\neq\res_{\fs}(k)$
for some integer $1\le k\le n$, then
$$\frac{x_k-\res_\fu(k)}{ \res_\fs(k)-\res_\fu(k)}m_{\fs\ft}=m_{\fs\ft}+\sum_{\fa\rhd\fs,
\fb}r_{\fa\fb}m_{\fa\fb} \quad \text{ for some } r_{\fa\fb}\in R.$$\end{remark}

(ii) Apply the $*$-anti-automorphism, we get that $$\quad\quad\quad  m_{\fs\ft}x_k=\res_\ft(k)m_{\fs\ft}
       +\sum_{\fa\rhd\ft}r_{\fa} m_{\fs\fa}
           \bmod\sh^{\rhd\gl} \quad \text{ for some } r_{\fa}\in R.$$

Let $\mcr(k)$ be
the complete set of possible residues $\res_\ft(k)$ as $\ft$ runs
over the set of all standard tableaux. Note that the residues in $\mcr(k)$ are all distinct.
\begin{definition}
[\protect{\cite[DEFINITION 6.7]{AMR}}]
Suppose that $\gl$ is an $m$-multipartition of $n$ and that $\fs$ and $\ft$ are standard $\lambda$-tableaux.
\begin{enumerate}
\item Let
$F_\ft:=\displaystyle\prod_{k=1}^n
  \displaystyle\prod_{\,\mcr(k)\ni c\neq\res_\ft(k)}
         \frac{x_k-c}{\res_\ft(k)-c}$.
\item Let $f_{\fs\ft}:=F_\fs m_{\fs\ft}F_\ft$.
\end{enumerate}\label{def F_t f_st}\end{definition}

Note that all of the factors in $F_\ft$ commute so there are no need to specify an order of the terms in the product.
The elements $F_{\ft}$ are defined for any choice of field $R$, regardless of whether or not Assumption~\ref{ppoly} holds.

 Extend the
dominance order to the pairs of standard tableaux by defining
$(\fs,\ft)\unrhd(\fa,\fb)$ if $\fs\unrhd\fa$ and $\ft\unrhd\fb$, and write $(\fs,\ft)\rhd(\fa,\fb)$
 if $(\fs,\ft)\unrhd(\fa,\fb)$ and $(\fs,\ft)\neq (\fa,\fb)$.

\begin{proposition}
Assume that Assumption~\ref{ppoly} holds and that $\gl$ is an $m$-multipartion of $n$. Suppose that $\fs$ and $\ft$
are standard $\lambda$-tableaux.
\begin{enumerate}
\setlength\itemsep{1\jot}
\item
$f_{\fs\ft}=m_{\fs\ft}
             +\displaystyle\sum_{(\fa,\fb)\rhd(\fs,\ft)}r_{\fa\fb}m_{\fa\fb}$
for some $r_{\fa\fb}\in R$;

\item If $\fu$ is a standard $\gl$-tableau, then $F_\fu f_{\fs\ft}=\delta_{\fs\fu}f_{\fs\ft}$ and
$f_{\fs\ft}F_\fu=\delta_{\ft\fu}f_{\fs\ft}$;

\item If $k$ is an integer with $1\le k\le n$, then  $x_kf_{\fs\ft}=\res_\fs(k)f_{\fs\ft}$.

\item If $\fa$ and $\fb$ are standard $\gl$-tableaux, then $f_{\fs\ft}f_{\fa\fb}=
\delta_{\fa\ft}r_{\ft}f_{\fs\fb}$ for some $r_{\ft}\in R$.
\end{enumerate}\label{Prop-f_st properties}
\end{proposition}
\begin{proof}  (i) follows directly from Theorem~\ref{x_k action}(i), Remark~\ref{Remark x_k m_st}, and
Definition~\ref{def F_t f_st}. Indeed, we have
$$\begin{aligned}f_{\fs\ft}&=F_{\fs}m_{\fs\ft}F_{\ft}\\&=\biggl(\prod_{k=1}^n
  \prod_{\,\mcr(k)\ni c\neq\res_\ft(k)}
         \frac{x_k-c}{\res_\fs(k)-c}m_{\fs\ft}\biggr)F_{\ft}\\
         &=\biggl(m_{\fs\ft}+\sum_{\fa\rhd\fs}r_{\fa}m_{\fa\ft}\biggr)\prod_{k=1}^n
  \prod_{\,\mcr(k)\ni c\neq\res_\ft(k)}\frac{x_k-c}{\res_\ft(k)-c}\\
  &=m_{\fs\ft}+\sum_{(\fa, \fb)\rhd(\fs, \ft)}r_{\fa\fb}m_{\fa\fb}.\end{aligned}$$

 (ii) and (iii) can be proved by using a variation argument of the proof of \cite[Proposition 3.35]{Mathas-book}.
  For the convenient of the reader especially for myself, we contain the proof. By (i) $f_{\fs\ft}=m_{\fs\ft}+\displaystyle\sum_{(\fa,\fb)\rhd(\fs,\ft)}r_{\fa\fb}m_{\fa\fb}$
for some $r_{\fa\fb}\in R$. First suppose that $\fu$ is a standard $\gl$-tableau such that $\fu\rhd \fs$,
 by Lemma~\ref{Lem res s t}(ii), there exists an integer $k_1$ such that $\res_{\fu}(k_1)\neq \res_{\fs}(k_1)$.
 Then $x_{k_1}-\res_{\fu}(k_1)$ is a factor of $F_{\fu}$, so $F_{\fu}m_{\fs\ft}=\sum_{\fb\rhd \fs}r_{\fb}m_{\fb\ft}$
  for some $r_{\fb}\in R$ by Theorem~\ref{x_k action}(i). Extended the dominance order $\unrhd$ to the total order $>$
  on the set of all standard $\gl$-tableaux, and chose $\fc_2$ with respect to $>$ with $r_{\fc_2}\neq 0$.
  Then $\fc_{2}\rhd \fs$
  and, as before, there exists an integer $k_2$ such that $\res_{\fc_2}(k_1)\neq \res_{\fs}(k_2)$.
  Therefore $F_{\fu}^2m_{\fs\ft}=\sum_{\fc\rhd\fb\rhd\fs}m_{\fc\ft}$. Continuing in this way shows that
  $F_{\fu}^Nm_{\fs\ft}=0$ whenever $\fu\rhd\fs$, where $N=\sum_{\gl}|\std(\gl)|$.
   So $F_{\fu}^Nf_{\fs\ft}=\delta_{\fu\fs}f_{\fs\ft}$. Similarly, $f_{\fs\ft}F_\fu^N=\delta_{\ft\fu}f_{\fs\ft}$ by using
   Remark~\ref{Remark x_k m_st}(ii).

   Next let $\tilde{f}_{\fs\ft}=F_{\fs}^Nf_{\fs\ft}F_{\ft}$. Then
   $$\begin{aligned}x_{k}\tilde{f}_{\fs\ft}=x_{k}F_{\fs}^Nf_{\fs\ft}F_{\ft}&=F_{\fs}^n(x_kf_{\fs\ft})\\
   &=F_{\fs}^n(\res_{\fs}(k)f_{\fs\ft}+\sum_{\fa\rhd\fs}r_{\fa}m_{\fs\ft})\\
   &=\res_{\fs}(k)\tilde{f}_{\fs\ft}.\end{aligned}$$
   Hence, $F_s\tilde{f}_{\fs\ft}=\tilde{f}_{\fs\ft}$. Furthermore, if $\fu\neq \fs$ then, by Lemma~\ref{Lem res s t}(ii),
   there exists an integer $k$ such that $\res_{\fu}(k)\neq \res_{\fs}(k)$. So
   $(x_k-\res_{\fu}(k))\tilde{f}_{\fs\ft}=0$ and $F_{\fu}\tilde{f}_{\fs\ft}=0$ because $x_k-\res_{\fu}(k)$ is
   a factor of $F_{\fu}$.

   Finally, note that there exist $r_{\fa}\in R$ such that $m_{\fs\ft}=\tilde{f}_{\fs\ft}+r_{\fa}\tilde{f}_{\fa\ft}$. So
   $f_{\fs\ft}=F_{\fs}m_{\fs\ft}F_{\ft}=F_{\fs}(\tilde{f}_{\fs\ft}+r_{\fa}\tilde{f}_{\fa\ft})F_{\ft}=\tilde{f}_{\fs\ft}$.
 Thus (ii) and (iii) are proved.

 (iv) By definition and (iii),

  \centerline{$\begin{aligned}
 f_{\fs\ft}f_{\fa\fb}=F_{\fs}m_{\fs\ft}F_{\ft}f_{\fa\fb}=\delta_{\fa\ft}F_{\fs}m_{\fs\ft}m_{\ft\fb}F_{\fb}
 =\delta_{\fa\ft}F_{\fs}\biggl(r_{\ft}m_{\fs\fb}+\sum_{(\fs', \fb')\rhd(\fs, \fb)}m_{\fs'\fb'}\biggr)F_{\fb}=\delta_{\fa\ft}r_{\ft}f_{\fs\fb}.
 \end{aligned}$}
\end{proof}

Now the ``Young's seminormal form" for Specht modules can be given as follows.

\begin{proposition}\label{Prop-s_i f_su}
Suppose that $\gl$ is a multipartition of $n$ and that $\fs$ and $\fu$ are standard
$\lambda$-tableaux. Let $\ft=\fs(i,i+1)$  for some integer $i$ with $1\le i<n$.

\begin{enumerate}
\item If $\ft$ is
standard then
\begin{align*}
s_if_{\fs\fu}&=\begin{cases}
    \displaystyle\frac{1}{\res_\ft(i)-\res_{\fs}(i)}f_{\fs\fu}
       +f_{\ft\fu},& \If \fs\rhd\ft,\\[5pt]
    \displaystyle\frac{1}{\res_\ft(i)-\res_{\fs}(i)}f_{\fs\fu}
    +\displaystyle\frac{(\res_\ft(i)-\res_\fs(i)-1)(\res_\ft(i)-\res_\fs(i)+1)}
         {(\res_\ft(i)-\res_{\fs}(i))^2}f_{\ft\fu},& \If \ft\rhd\fs.
   \end{cases}\end{align*}

\item If $\ft$ is not standard then
\begin{align*}s_if_{\fs\fu}&=\begin{cases}
  f_{\fs\fu}, &\text{if $i$ and $i+1$ are in the same row of $\fs$},\\
  -f_{\fs\fu}, &\text{if $i$ and $i+1$ are in the same column of $\fs$.}
\end{cases}\end{align*}\end{enumerate}
\label{Tf multiplication}
\end{proposition}

\begin{proof}
By Theorem~\ref{std basis} and Proposition~\ref{Prop-f_st properties}(i), $\{f_{\fs\fu}\}$ is a basis of $\sh$, so
$s_if_{\fs\fu}=\sum_{\fa,\fb}r_{\fa\fb}f_{\fa\fb}$ for some $r_{\fa\fb}\in R$.
By Proposition~\ref{Prop-f_st properties}(ii), $ f_{\fa\fb}F_\fu=\delta_{\fb\fu}f_{\fa\fu}$.
Therefore, multiplying the equation for $s_if_{\fs\fu}$ on the right by
$F_\fu$ shows that $r_{\fa\fb}=0$ whenever $\fa\neq\fu$; in particular,
$r_{\fa\fb}=0$ if $\shape(\fb)\ne\lambda$. Hence,
$s_if_{\fs\fu}=\sum_\fa r_\fa f_{\fa\fu}$, for some $r_\fa\in R$, where $\fa$
runs over the set of standard $\lambda$-tableaux.

Suppose that $k$ is an integer such that $k\neq i, i+1$.
Then, by Lemma~\ref{Lem si xj}(ii) and Proposition~\ref{Prop-f_st properties}(iii),
  \begin{align*}x_ks_if_{\fs\fu}=s_ix_kf_{\fs\fu}=\res_{\fs}(k)s_if_{\fs\fu}=\res_{\fs}(k)
  \sum_{\fa\in\std(\gl)}r_{\fa}f_{\fa\fu},\end{align*}
  On the other hand, by Proposition~\ref{Prop-f_st properties}(iii), we have
   \begin{align*}x_ks_if_{\fs\fu}=\sum_{\fa\in\std(\gl)}r_{\fa}x_kf_{\fa\fu}=\sum_{\fa\in\std(\gl)}
  r_{\fa}\res_{\fa}(k)f_{\fa\fu}.\end{align*}
  Equating coefficients, $r_{\fa}\res_{\fs}(k)=r_{\fa}\res_{\fa}(k)$ for all $k\neq i, i+1$,
   $\fa\in\std(\gl)$.
   Therefore, by Lemma~\ref{s=t or s=t(i i+1)}(ii), $r_{\fa}=0$ unless either $\fa=\fs$ or
    $\fa=\ft$ and $\ft$ is standard.

   Suppose that $\ft$ is not standard . Then we have shown $s_if_{\fs\fu}=r_{\fs}f_{\fs\fu}$
    for some $r_{\fs}\in R$.
   By Proposition~\ref{Prop-f_st properties}(i), $f_{\fs\fu}=m_{\fs\fu}+\sum_{(\fa, \fb)\rhd(\fs, \fu)}r_{\fa\fb}f_{\fa\fb}
   $ for some $r_{\fa\fb}\in R$.
   Because $\ft$ is not standard, either $i$ and $i+1$ are in the same row of $\ft$ or
    they are in the column. In the first case, by \cite[Lemma 1.1(iv)]{DJ:reps},
    $s_i\in S_n\cap S_{\gl}$ and $d(\fs)=s_id(\fs)s_i$. Therefore
    $s_im_{\fs\fu}=s_id(s)^*m_{\gl}d(\fu)=d(s)^*x_{\gl}u^{+}_{\gl}d(\fu)=m_{\fs\fu}$.

   In the second case, there is a unique standard tableau $\fc$ such that $\ft=\fs (i, i+1)=\fc s_jw $
   for some $j$ and $w\in S_n$ with $\ell(\ft)=\ell(\fc)+1+\ell(w)$, and if $\fb$ is any (standard) tableau with
   $\fb\rhd\fc s_j$ then $\fb\rhd\fc$. By construction, $\fs=\fc(s_jws_i)$ and $\ell(\fs)=\ell(\fc)+\ell(w)$. Therefore
   $\ell(w)=\ell(s_jws_i)$. Similarly, $\ell(ws_i)=\ell(w)+1=\ell(s_jw)$. Therefore $w=s_jws_i$,
      \begin{align*}&s_im_{\fs\fu}=s_id(s)^*m_{\gl}d(\fu)=w^*s_jd(\fc)^*m_{\gl}d(\fu)=w^*m_{\fc s_j}, \qquad\text{ and }\\
  &m_{\fc{s_j}\fu}=m_{\fc{s_j}\ft^{\gl}}d(\fu)^*=-m_{\fc{s_j}\fu}-
   \sum_{b\rhd \fc}m_{\fb\fu} \bmod \sh^{\rhd\gl}.\end{align*}
   Therefore
   \begin{align*}s_im_{\fs\fu}=w^*m_{\fc{s_j}\fu}&=
   -(s_iw^*s_j)m_{\fc{s_j}\fu}-\sum_{\fb\rhd\fc}s_iw^*s_jm_{\fb\fu}+h\quad
   \text{for some }h\in\sh m_{\mu}\sh \text{ and }\mu\rhd\gl\\
   &=-s_iw^*d(\fc)^*m_{\ft^{\gl}\fu}-\sum_{\fb\rhd\fc}s_iw^*s_jd(\fb)^*m_{\ft^{\gl}\fu}+h\\
   &=-m_{\fs\fu}+\sum_{\fb\rhd\fs}r_{\fb}m_{\fb\fu}.\end{align*}
   So $s_if_{\fs\fu}=-f_{\fs\fu}$ in the second case.

   Now suppose that $\ft=\fs(i,i+1)$ is standard, we have shown that
   $s_if_{\fs\fu}=r_{\fs}f_{\fs\fu}+r_{\ft}f_{\ft\fu}$ for some $r_{\fs}, r_{\ft} \in R$.
   First suppose that $\fs\rhd\ft$. Then $s_im_{\fs\fu}=m_{\ft\fu}$ for any $\fu\in\std(\gl)$ since $d(\ft)=d(\fs)s_i$.
   Therefore
    \begin{align*}
   s_if_{\fs\fu}&=s_i(m_{\fs\fu}+\sum_{(\fa, \fb)\rhd(\fs, \fu)}r_{\fa\fb}f_{\fa\fb})
   \text{ for some } r_{\fa\fb}\in R\\
   &=m_{\ft\fu}+\sum_{(\fa, \fb)\rhd(\fs, \fu)}r_{\fa\fb}s_if_{\fa\fb}\\
   &=m_{\ft\fu}+\sum_{(\fa, \fb)\rhd(\ft, \fu)}r_{\fa \fb}f_{\fa \fb}.
   \end{align*}
   Hence $r_{\ft}=1$, that is, $s_if_{\fs\fu}=r_{\fs}f_{\fs\fu}+f_{\ft\fu}$.
   Now, by Lemma~\ref{Lem si xj}(ii), we get
        \begin{align*}&x_{i+1}(s_if_{\fs\ft})=r_{\fs}x_{i+1}f_{\fs\fu}+x_{i+1}f_{\ft\fu}=
   r_{\fs}\res_{\fs}(i+1)f_{\fs\fu}+\res_{\ft}(i+1)f_{\ft\fu}, \text{ and }\\
    &(x_{i+1}s_i)f_{\fs\fu}=s_ix_if_{\fs\fu}+f_{\fs\fu}=
   (r_{\fs}\res_{\fs}(i)+1)f_{\fs\fu}+\res_{\fs}(i)f_{\ft\fu}.\end{align*}
Note that $\res_{\fs}(i)=\res_{\ft}(i+1)$ and  $\res_{\fs}(i+1)=\res_{\ft}(i)$,  we yield that
 $r_{\fs}=\frac{1}{\res_{\ft}(i)-\res_{\fs}(i)}$.

   Suppose that $\ft\rhd\fs$. Then $\ft\rhd\ft(i,i+1)=\fs$ and
   $s_if_{\ft\fu}=\frac{1}{\res_{\fs}(i)-\res_{\ft}(i)}f_{\ft\fu}+f_{\fs\fu}$ by the same argument as above. Thus
   \begin{align*}f_{\fs\fu}=s_i^2f_{\fs\fu}=s_i(r_{\fs}f_{\fs\fu}+r_{\ft}f_{\ft\fu})
   =(r_{\fs}^2+r_{\ft})f_{\fs\fu}+r_{\ft}(r_{\fs}-\frac{1}{\res_{\ft}(i)-\res_{\fs}(i)})f_{\ft\fu},\end{align*}
  which implies that $r_{\fs}=\displaystyle\frac{1}{\res_{\ft}(i)-\res_{\fs}(i)}$ and
   $r_{\ft}=\displaystyle\frac{(\res_\ft(i)-\res_\fs(i)-1)(\res_\ft(i)-\res_\fs(i)+1)}
         {(\res_\ft(i)-\res_{\fs}(i))^2}$.
 \end{proof}

Our next step is to construct an orthogonal basis of Specht modules with respect to the bilinear form $\la\,\ra$.
For each standard $\lambda$-tableau $\fs$ let
$f_\fs=f_{\fs\ft^{\gl}}+\sh^{\rhd\gl}$. We have the following facts.

\begin{corollary}\label{Cor orthogonal basis of Spechts}
 Assume that Assumption~\ref{ppoly} holds. Suppose that $\gl$ is an  $m$-multipartition of $n$.
\begin{enumerate}
\item Suppose that $\ft$ is a standard $\gl$-tableau.
\begin{enumerate}
\item There exist $r_{\fs}\in R$ such that $f_{\ft}=m_{\ft}+\sum_{\fs\rhd\ft}r_{\fs}m_{\fs}$.
\item Suppose that $k$ is an integer with $1\leq k\leq n$. Then $x_kf_{\ft}=\res_{\ft}(k)f_{\ft}$.
\item Suppose that $\fs$ is a standard $\gl$-tableau. Then $F_{\fs}f_{\ft}=\delta_{\fs\ft}f_{\ft}$.
\end{enumerate}
\item Suppose that $\fs$ and $\ft$ are standard
$\lambda$-tableaux.
\begin{enumerate}\item   If $\fs=\ft(i,i+1)\rhd\ft$ then
$f_\ft=(s_i-\alpha)f_\fs$, where
$\alpha=\frac{1}{\res_\ft(i)-\res_\fs(i)}$.
\item $\la f_{\fs}, f_{\ft}\ra=\delta_{\fs\ft}r_{\ft}$ for some $r_{\ft}\in R$.
\item $\{f_\fs|\fs\in\std(\lambda)\}$ is
an orthogonal basis of the Specht module $S^\lambda$.
\end{enumerate}
\end{enumerate}
\end{corollary}
\begin{proof}(i) follows directly by Proposition~\ref{Prop-f_st properties}.
Theorem~\ref{Prop-s_i f_su}(i) implies (ii.a).  Now, by
 Proposition~\ref{Prop-f_st properties}(i) and (iv), (ii.b) is proved. (ii.c) is proved by using
 Theorem~\ref{std basis} and Proposition~\ref{Prop-f_st properties}(i) and (iv).\end{proof}

The inner products $\la f_\fs,f_\ft\ra$, for $\fs,\ft\in\std(\lambda)$ will be
computed explicitly (as rational functions) in  the following.
To describe this we need some more notation.    Given
two nodes $x=(i,j,k)$ and $y=(a,b,c)$, write $y\prec x$, if either $c<k$, or
$c=k$ and $b>j$.

Suppose that $\gl$ is an $m$-multipartition of $n$ and that $\fs$ be a standard $\lambda$-tableau.
Then for each integer $i$ with $1\le i\le n$ there is a
unique node $x\in[\lambda]$ such that $\fs(x)=i$.  Let
$\mathscr A_\fs(i)$ be the set of addable nodes for the $m$-multipartition $\shape(\fs\!\downarrow\!i)$
which are strictly greater than $x$ (with respect to $\prec$); similarly, let
$\mathscr R_\fs(i)$ be the set of removable nodes which are strictly greater than
$x$ for the $m$-multipartition $\shape(\fs\!\downarrow\!i-1)$. If $y=(i, j, s)$  is either
an addable or a removable node, then we define its residue to be $\res(y)=j-i+q_s$.
Finally, if $\lambda$ is an
$m$-multipartition let
$\lambda!=\prod_{s=1}^m\prod_{i\geq1}\lambda^{s}_i!$.

\begin{example}Let $\gl=(3\cdot1; 1)$. Then $\sa_{\ft^{\gl}}(1)=\{(1,1,2)\}$, $\mathscr{R}_{\ft^{\gl}}(1)=\{\emptyset\}$; $\sa_{\ft^{\gl}}(2)\!=\!\{(2,1,1),(1,1,2)\}$, $\mathscr{R}_{\ft^{\gl}}(2)\!=\!\{(1,1,1)\}$; $\sa_{\ft^{\gl}}(3)\!=\!\{(2,1,1),(1,1,2)\}$, $\mathscr{R}_{\ft^{\gl}}(3)\!=\!\{(1,2,1)\}$; $\sa_{\ft^{\gl}}(4)\!=\!\{(1,1,2)\}$, $\mathscr{R}_{\ft^{\gl}}(4)=\{\emptyset\}$;  $\sa_{\ft^{\gl}}(5)=\{\emptyset\}=\mathscr{R}_{\ft^{\gl}}(5)$. It follows directly that
$$\displaystyle\prod_{i=1}^5
   \frac{\prod_{x\in\mathscr A_{\ft^{\gl}}(i)} (\res_{\ft^{\gl}}(i)-\res(x))}%
        {\prod_{y\in\mathscr R_{\ft^{\gl}}(i)} (\res_{\ft^{\gl}}(i)-\res(y))}=3!(-1+q_1-q_2)(q_1-q_2)(1+q_1-q_2)(2+q_1-q_2),$$ which equals exactly to  $\gma_{\ft^{\gl}}=\la f_{\ft^{\gl}}, f_{\ft^{\gl}}\ra$ determined by Theorem~\ref{gamma properties}(ii.a).
\end{example}

Now we can give a closed formula of $\gamma_{\ft}:=\la f_{\ft}, f_{\ft}\ra$ for any $m$-multipartition $\gl$ of
$n$ and any standard $\gl$-tableau $\ft$. For a moment, we write $i=(a, b,c)\in \ft$ if the integer $i$ with $1\le i\le n$ appears in the unique node $(a, b, c)$ of $[\gl]$ such that $\ft(a, b,c)=i$.
\begin{theorem}
Assume that Assumption~\ref{ppoly} holds. Suppose that $\gl$ is an $m$-multipartition of $n$.
\begin{enumerate}
\item Suppose that $\ft$ is a standard $\gl$-tableau. Then $\gamma_\ft$ is uniquely determined by the two conditions
\begin{enumerate}
\item $\gamma_{\ft^{\gl}}=
            \lambda!\displaystyle\prod_{1\le s<t\le m}
             \displaystyle\prod_{(i,j)\in[\lambda^{s}]}(j-i+q_s-q_t)$; and
\item if $\fs=\ft(i,i+1)\rhd\ft$ then
$\gamma_\ft=\displaystyle\frac{(\res_\ft(i)-\res_\fs(i)+1)(\res_\ft(i)-\res_\fs(i)-1)}
                {(\res_\ft(i)-\res_{\fs}(i))^2}\gamma_{\fs}$.
\end{enumerate}

\vspace{1\jot}
\item Let $\fs$ be a standard $\gl$-tableau. Then $\gamma_\fs=\displaystyle\prod_{i=1}^n
   \frac{\prod_{x\in\mathscr A_\fs(i)} (\res_\fs(i)-\res(x))}
        {\prod_{y\in\mathscr R_\fs(i)} (\res_\fs(i)-\res(y))}$.
\end{enumerate}
\label{gamma properties}
\end{theorem}
\begin{proof}
(i.a) By definition,
\begin{align*}f_{\ft^{\gl}}f_{\ft^{\gl}}=
(F_{\ft^{\gl}}m_{\gl}F_{\ft^{\gl}})(F_{\ft^{\gl}}m_{\gl}F_{\ft^{\gl}})=m_{\gl}m_{\gl}&=\gl!\prod_{t=2}^m\prod_{k=1}^{a_t}(x_k-q_l)m_{\gl}\\
&=\gl!\prod_{t=2}^m\prod_{k=1}^{a_t}(\res_{\ft^{\gl}}(k)-q_t)m_{\gl}\\
&=\gl!\prod_{t=2}^m\prod_{\substack{s<t\\(i, j)\in[\gl^s]}}(j-i+q_s-q_t)m_{\gl}\\
&=\gamma_{\ft}m_{\gl}.\end{align*}
So (ii.a) is proved.

Suppose that $\fs=\ft(i,i+1)\rhd\ft$ and let $\alpha=\frac{1}{\res_\ft(i)-\res_\fs(i)}$.
 Applying Corollary~\ref{Cor orthogonal basis of Spechts}(ii), $f_\ft=(s_i-\alpha)f_\fs$, moreover  $\la s_if_{\fs}, f_{\fs}\ra=\ga\gma_{\fs}$ and $\la s_if_{\fs},s_if_{\fs}\ra=\gma_{\fs}$ since $\la f_{\fs}, f_{\ft}\ra=0$. Thus
\begin{align*}\gamma_{\ft}=\la s_if_{\fs}, s_if_{\fs}\ra-2\ga\la s_if_{\fs}, f_{\fs}\ra+\ga^2\gma_{\fs}
 =(1-\ga^2)\gma_{\fs}.\end{align*}
So (i.b) is proved.

 (ii) We proceed by induction on $\ell(\fs)$ for standard $\gl$-tableaux $\fs$.  First, let $\fs=\ft^{\gl}$ and assume that the integer $i=(a, b,c)$ with $1\le i\le n$, that is $a_c<i\le a_{c+1}$. First consider the contribution that the addable and removable nodes in $[\gl^c]$ make to $\ft^{\gl}$. By definition, these nodes occur in pairs $(x, y)=((a+1,1,c), (a, b-1,c)$. Therefore
  $$\frac{\prod_{x\in\sa_{\ft^{\gl}}(i)\cap[\gl^c]}(\res_{\ft^\gl}(i)-\res(x))}
  {\prod_{y\in\mathscr{R}_{\ft^{\gl}}(i)\cap[\gl^c]}(\res_{\ft^\gl}(i)-\res(y))}=b \quad\text{ and }\quad\prod_{a_c+1\le i\le a_{c+1}}\frac{\prod_{x\in\sa_{\ft^{\gl}}(i)\cap[\gl^c]}(\res_{\ft^\gl}(i)-\res(x))}
  {\prod_{y\in\mathscr{R}_{\ft^{\gl}}(i)\cap[\gl^c]}(\res_{\ft^\gl}(i)-\res(y))}=\gl^c!.$$
  Next, consider the contribution that addable and removable nodes in some $t$-component of $[\gl]$ with $t>c$ (there are no such node for $t<c$). In this case, we have $\mathscr{A}_{\ft^{\gl}}(i)=\{(1, 1, c+1), \cdots, (1, 1, m)\}$ and $\mathscr{R}_{\ft^{\gl}}(i)=\{\emptyset\}$. Therefore
  \begin{align*}\prod_{a_c+1\le i\le a_{c+1}}\frac{\prod_{x\in\sa_{\ft^{\gl}}(i)\backslash[\gl^c]}(\res_{\ft^\gl}(i)-\res(x))}
  {\prod_{y\in\mathscr{R}_{\ft^{\gl}}(i)\backslash[\gl^c]}(\res_{\ft^\gl}(i)-\res(y))}=\prod_{(i, j)\in[\gl^c]}\prod_{t=c+1}^m(j-i+q_c-q_t).\end{align*}
  Hence
  \begin{align*}\prod_{i=1}^n
   \frac{\prod_{x\in\mathscr A_{\ft^{\gl}}(i)} (\res_{\ft^{\gl}}(i)-\res(x))}
        {\prod_{y\in\mathscr R_{\ft^{\gl}}(i)} (\res_{\ft^{\gl}}(i)-\res(y))}=\gl!\prod_{1\le s<t\le m}\prod_{(i, j)\in[\gl^c]}(j-i+q_s-q_t)=\gma_{\ft^{\gl}}.\end{align*}
\indent Now assume that $\ft^{\gl}\rhd\fs$ and there exists an integer $k$ with $1\le k<n$ such that $\ft^{\gl}=\fs(k,k+1)$.
 let $\ga=\frac{1}{\res_{\fs}(k)-\res_{\ft^{\gl}}(k)}$ and let
\begin{align*}L:=\prod_{i=k, k+1}\frac{\prod_{x\in\mathscr{A}_{\fs}(i)} (\res_{\fs}(i)-\res(x))}
        {\prod_{y\in\mathscr R_{\fs}(i)} (\res_{\fs}(i)-\res(y))} \quad\text{ and }\quad R:=\prod_{i=k, k+1}
        \frac{\prod_{x\in\mathscr{A}_{\ft^{\gl}}(i)} (\res_{\ft^{\gl}}(i)-\res(x))}
        {\prod_{y\in\mathscr R_{\ft^{\gl}}(i)} (\res_{\ft^{\gl}}(i)-\res(y))}.\end{align*}
Since $\mathscr{A}_{\fs}(i)=\mathscr A_{\ft^{\gl}}(i)$ and $\mathscr{R}_{\fs}(i)=\mathscr{R}_{\ft^{\gl}}(i)$ for all $i\neq k, k+1$.  Therefore, by (ii.b) and Assumption~\ref{ppoly}, we only need to show that
$L=(\ga^2-1)R$.

Note that either both $k$ and $k+1$ appear in the same component of $\fs$, say the $c$-component of $\fs$, or $k$ and $k+1$ appear in the different components of $\fs$. In the first case, we have that $a_c<k, k+1\le a_{c+1}$, moreover, $k=(a+1,1,c)\in\fs$ and $k+1=(a,\gl^c_a,c)\in\fs$ for some $a$ with $\gl^c_a>1$ and $a<\ell(\gl^c)$. Now by definition,
\begin{align*}&\sa_{\fs}(k)=\{(1,1,c+1), \dots, (1, 1, m)\},\quad\quad\sr_{\fs}(k)=\{\emptyset\};\\ &\sa_{\fs}(k+1)=\{(a+1, 2,c), (a+2, 1, c),(1,1,c+1), \dots, (1, 1, m)\},\\& \sr_{\fs}(k+1)=\{(a+1, 1, c), (a, \gl^c_a-1, c)\}.\end{align*}
On the other hand, since $k=(a, \gl^c_a, c)\in\ft^{\gl}$ and $k=(a+1, 1, c)\in\ft^{\gl}$,
\begin{align*}&\sa_{\ft^{\gl}}(k)=\{(a+1,1, c), (1,1,c+1), \dots, (1, 1, m)\},\quad\quad\sr_{\ft^{\gl}}(k)=\{(a,\gl^c_a-1, c)\};\\
&\sa_{\ft^{\gl}}(k+1)=\{(1,1,c+1), \dots, (1, 1, m)\}, \quad\quad\quad\quad\quad\,\sr_{\fs}(k+1)=\{\emptyset\}.\end{align*}
Hence $\ga^2-1=\frac{(\gl^c_a)^2-1}{(\gl^c_a)^2}$ and  \begin{align*}&L=\frac{(\gl^c_a)^2-1}{\gl^c_a}\prod_{j=1}^{m-c}(-a+q_c-q_{c+j})(\gl^c_a-a+q_c-q_{c+j}),\\
&R=\gl^c_a\prod_{j=1}^{m-c}(-a+q_c-q_{c+j})(\gl^c_a-a+q_c-q_{c+j}).\end{align*}
It follows directly that $L=(\ga^2-1)R$.

In the second case, clearly $k=a_{c+1}$ for some $c$ with $1\le c<m$. Therefore $k+1=(a,\gl^c_a,c)\in\fs$ where $a=\ell(\gl^c)$ and  $k=(1, 1,d)\in\fs$ where $d=\min\{m\geq d>c\mid \gl^{d}\neq\emptyset\}$.  Then $\res_{\fs}(k)=q_d$, $\res_{\fs}(k+1)=\gl^c_a-a+q_c=\res_{\ft^{\gl}}(k)$, and
$$\begin{aligned}&\sa_{\fs}(k)=\{(1,1,d+1), \dots, (1, 1, m)\}, \quad\quad\sr_{\fs}(k)=\{\emptyset\}; \\
&\sa_{\fs}(k+1)=\{(a+1, 1,c), (1, 1, c+1), \dots, (1, 1, d-1),  (1, 2, d),(2,1,d), (1, 1, d+1),  \dots, (1, 1, m)\},\\
&\sr_{\fs}(k+1)=\{(a, \gl^c_a-1, c), (1, 1, d)\}.\end{aligned}$$
On the other hand,   since $k=(a, \gl^c_a, c)\in\ft^{\gl}$ and $k=(1, 1, d)\in\ft^{\gl}$,
$$\begin{aligned}&\sa_{\ft^{\gl}}(k)=\{(a+1,1, c), (1,1,c+1), \dots, (1, 1, m)\}, \quad\quad\sr_{\ft^{\gl}}(k)=\{(a,\gl^c_a-1, c)\};\\
&\sa_{\ft^{\gl}}(k+1)=\{(1,1,d+1), \dots, (1, 1, m)\},\quad\quad\quad\quad\,\quad\sr_{\fs}(k+1)=\{\emptyset\}.\end{aligned}$$
Therefore \begin{align*}&\ga^2-1=\frac{\gl^c_a(\gl^c_a-a+1+q_c-q_d)(\gl^c_a-a-1+q_c-q_{d})}{(\gl^c_a-a+q_c-q_{d})^2},\\
&L=\frac{\gl^c_a(\gl^c_a-a+1+q_c-q_d)(\gl^c_a-a-1+q_c-q_{d})}{\gl^c_a-a+q_c-q_{d}}
\prod_{j=1}^{m-d}(q_d-q_{d+j})(\gl^c_a-a+q_c-q_{d+j}),\\
&R=\gl^c_a\prod_{j=0}^{m-d}(\gl^c_a-a+q_c-q_{d+j})\prod_{j=1}^{m-d}(q_d-q_{d+j}).\end{align*}
It follows directly that $L=(\ga^2-1)R$ and we complete the proof by induction argument.
  \end{proof}

\begin{remark} The assertion (i) has given in \cite[LEMMA 6.10]{AMR} with a skeleton outline of proof. The assertion (ii) answers a question of Ariki-Mathas-Rui \cite[\S6.9]{AMR}.
\end{remark}

Let $(,)$ be the inner product on $\sh$ given by
$(h_1,h_2)=\tau(h_1h_2^*)$, for $h_1,h_2\in\sh$. It follows form Theorem~\ref{Thm-def tau} and
\S\ref{*-anti-auto} that the trace form $\tau$ satisfying
that $\tau(h)=\tau(h^*)$ for all $h\in \sh$. Therefore $(\,,)$ is a
symmetric bilinear form on $\sh$. Furthermore, $(\,,)$ is associative in the sense that
$(ab, c)=(a, cb^*)$ for all $a, b, c\in\sh$.

\begin{theorem}[\cite{AMR}, PROPOSITION 6.8]
Assume that Assumption~\ref{ppoly} holds.
\noindent\begin{enumerate}\item If $\fs,\ft,\fa$ and $\fb$ are standard
tableaux then $f_{\fs\ft}f_{\fa\fb}=\delta_{\fa\ft}\gamma_\ft f_{\fs\fb}$.

\item $\{f_{\fs\ft}|\fs,\ft\in\std(\lambda),\lambda\in\mpn\}$
is a orthogonal basis of $\sh$ with respect to the trace
form $\tau$. \end{enumerate}
\label{Them orthogonal basis}
\end{theorem}
\begin{proof} By proposition~\ref{Prop-f_st properties}(iv), $f_{\fs\ft}f_{\fa\fb}=
\gd_{\fa\ft}r^{\ft}_{\fs\fb}f_{\fs\fb}$  for
some $r^{\ft}_{\fs\fb}\in R$. So we only need to show that $r^{\ft}_{\fs\fb}
=\gma_{\ft}$ for all standard tableaux $\fs, \fb$. Note that $f_{\ft}=f_{\ft\ft^{\gl}}+\sh^{\rhd\gl}$
 and $\gma_{\ft}=\la f_{\ft}, f_{\ft}\ra$. Therefore, the inner product $\la, \ra$ on the Specht module $S^\gl$
 gives $\gma_{\ft}m_{\gl}=\la f_{\ft}, f_{\ft}\ra m_{\gl}\equiv f_{\ft^{\gl}\ft}f_{\ft\ft^{\gl}}\bmod \sh^{\rhd\gl}$.
 Hence $f_{\ft^{\gl}\ft}f_{\ft\ft^{\gl}}=\gma_{\ft}f_{\ft^{\gl}\ft^{\gl}}$ since $f_{\ft^{\gl}\ft^{\gl}}=m_{\gl}$.
 Thus $$\begin{aligned}f_{\fs\ft}f_{\fa\fb}&=\gd_{\fa\ft}F_{\fs}d(\fs)^*m_{\ft^{\gl}\ft}F_{\ft}F_{\ft}m_{\ft\ft^{\gl}}d(\ft)F_{\fb}\\
 &=\gd_{\fa\ft}F_{\fs}d(\fs)^*f_{\ft^{\gl}\ft}f_{\ft\ft^{\gl}}d(\ft)F_{\fb}\\
 &=\gd_{\fa\ft}\gma_{\ft}F_{\fs}d(\fs)^*m_{\gl}d(\fb)F_{\fb}\\
 &=\gd_{\fa\ft}\gma_{\ft}f_{\fs\fb}.\end{aligned}$$
(i) is proved.

By Theorem~\ref{std basis} and Proposition~\ref{Prop-f_st properties}(i),
$\{f_{\fs\ft}|\fa,\ft\in \std(\gl), \gl\in\mpn\}$ is a basis of $\sh$. Now, by (i),
$(f_{\fs\ft}, f_{\fa\fb})=\tau(f_{\fs\ft}f_{\fa\fb}^*)=\tau(f_{\fs\ft}f_{\fb\fa})=
\gd_{\ft\fb}\gma_{\ft}\tau(f_{\fs\fa})$. On the other hand, $\tau$ is a trace form, so we also have that
$\tau(f_{\fs\ft}f_{\fb\fa})=\tau(f_{\fb\fa}f_{\fs\ft})=\gd_{\fa\fs}\gma_{\fs}\tau(f_{\fb\ft})$.
Thus $(f_{\fs\ft}, f_{\fa\fb})=
\gd_{\ft\fb}\gma_{\ft}\tau(f_{\fs\fa})=\gd_{\fa\fs}\gma_{\fs}\tau(f_{\fb\ft})=
\gd_{\ft\fb}\gd_{\fa\fs}\gma_{\ft}\tau(f_{\ft\ft})$. Consequently,
$\{f_{\fs\ft}|\fs,\ft\in\std(\lambda)\for \lambda\in\mpn\}$
is a orthogonal basis of $\sh$ with respect to the trace
form $\tau$ and $\tau(f_{\ft\ft})\neq 0$ for all standard tableaux $\ft$. Indeed, if there is a standard
 tableau $\ft$ such that $\tau(f_{\ft\ft})=0$. Then $\gma_{\ft}\tau(f_{\ft\ft})=\gma_{\fs}\tau(f_{\fs\fs})$ and
 $\gma_{\ft}\neq 0$ implies that $\tau_{f_{\fs\fs}}=0$ for all standard tableau $\fs$ since $\gma_{\fs}\neq 0$
 for all standard tableaux $\fs$. It is impossible since $\tau$ is non-degenerate. We complete the proof.
 \end{proof}

\begin{remark}
In \cite[Theorem A2]{BK-Schur-Weyl} it was shown that $\sh$ is a symmetric algebra with respect to the trace
form $\tau$ for all parameters $q_1, \dots, q_m$; however, this was proved indirectly without constructing
a pair of dual bases. The Theorem gives a self-dual basis of the
semisimple degenerate cyclotomic Hecke algebras. In general,  no such basis is known in general.
\end{remark}

Now we may identify  the Specht module $S^\lambda$ with a
submodule of $\sh$ up to isomorphism.

\begin{corollary}
Suppose that $\fs$ and
$\ft$ be standard $\lambda$-tableaux. Then $S^\lambda\cong
\sh f_{\fs\ft}=\sum_{\fa\in\std(\lambda)}Rf_{\fa\ft}$.
\label{S=fH}
\end{corollary}
\begin{proof}
The Theorem implies that$\{f_{\fa\ft}\mid \fa\in\std(\gl)\}$ is a basis of $\sh f_{\fs\ft}$.
On the other hand, by \S\ref{Def Specht} and Proposition~\ref{Prop-f_st properties}(i), $\{f_{\fa}\mid\fa\in
\std(\gl)\}$ is a basis of the Specht module $S^{\gl}$. Now the $R$-linear map $f_{\fa}\mapsto f_{\fa\ft}$
 gives the desired  isomorphism.
\end{proof}

Let $\mathcal{G}(\gl)=\det(\la m_{\fs}, m_{\ft}\ra)$,
for $\fs, \ft\in\std(\gl)$,  be the Gram determinant of this form, which
is well-defined up to a unit in $R$. As an application of the closed formula for $\gma_{\fs}$, we obtain a closed formula for the Gram determinant, which is differential but equivalent to that one given in \cite[COROLLARY 6.9]{AMR}.

\begin{corollary}Suppose that $R$ is a field and that Assumption~\ref{ppoly} holds. Let $\gl$ be an $m$-multipartition of $n$. Then \begin{align*}\mathcal{G}(\gl)=\prod_{\ft\in\std(\gl)}\gma_{\ft}=\prod_{\ft\in\std(\gl)}\prod_{i=1}^n
   \frac{\prod_{x\in\mathscr A_\ft(i)} (\res_\ft(i)-\res(x))}
        {\prod_{y\in\mathscr R_\ft(i)} (\res_\ft(i)-\res(y))}\end{align*}
\end{corollary}

\begin{proof} Fix $\ft\in\std(\lambda)$. Then Specht module
$S^{\lambda}$ is isomorphic to the submodule of
$\sh/\sh^{\rhd\lambda}$ which is spanned by
$\{m_{\fs\ft}+\sh^{\rhd\lambda}|\fs\in\std(\lambda)\}$, where
the isomorphism is given by
$\theta: \sh/\sh^{\rhd\lambda}\rightarrow S^{\gl};
          m_{\fs\ft}+\sh^{\rhd\lambda}\mapsto m_\fs$.
On the other hand, by Corollary~\ref{Cor orthogonal basis of Spechts}(ii.c) and (i.a), $\{f_\fs|\fs\in\std(\lambda)\}$ is a
basis of $S^{\gl}$ and the transition matrix between the two bases
$\{m_\fs\}$ and $\{f_\fs\}$ of $S^{\gl}$ is unitriangular. Consequently,
$\mathcal{G}(\lambda)=\det\(\langle f_\fs,f_\ft\rangle\)$, where
$\fs,\ft\in\std(\lambda)$. However, it follows from Theorem~\ref{Them orthogonal basis}(i) and Corollary~\ref{Cor orthogonal basis of Spechts}(ii.b) that
$\langle f_\fs,f_\ft\rangle=\delta_{\fs\ft}\gamma_\ft$. Hence the result follows directly by Theorem~\ref{gamma properties}(ii).
\end{proof}

Set $\tilde f_{\fs\ft}=\gamma_\ft^{-1}f_{\fs\ft}$.  Then
$\tilde f_{\fs\ft}\tilde f_{\fa\fb}=\delta_{\ft\fa}\tilde f_{\fs\fb}$ and
$\{\tilde f_{\fs\ft}\}$ is a basis of $\sh$. Hence, we have the
following.

\begin{corollary} Assume that Assumption~\ref{ppoly} holds. Then $\{\tilde f_{\fs\ft}|\fs,\ft\in\std(\lambda)\for \lambda\in\mpn\}$ is a bases of matrix units in $\sh$.
\end{corollary}

 The last result yields an explicit isomorphism from $\sh$ to the group
ring of $W_{m,n}$ when~$P_\sh(Q)$ is invertible.  Assume that $R$
contains a primitive $m$-th root of unity $\zeta$; then, $\sh\cong
RW_{m,n}$ when $q_s=\zeta^s$ for $s=1, \dots,m$. Write
$\tilde f^1_{\fs\ft}$ for the element of $RW_{m,n}$ corresponding to
$\tilde f_{\fs\ft}\in\sh$ under the canonical isomorphism
$\sh\rr RGW_{m,n}$.

\begin{corollary}\label{Cor: iso H G(m, 1,n)}
Assume that $R$ contains a primitive $m$-th root of unity and that
Assumption~\ref{ppoly} holds. Then $\sh\cong RW_{m,n}$
via the $R$-algebra homomorphism determined
by~$\tilde f_{\fs\ft}\mapsto \tilde f^1_{\fs\ft}$.
\end{corollary}

By parts (i) and (iii) of Theorem~\ref{Them idempotents} below,
$s_i=\sum_\ft \tilde f_{\ft\ft} s_i$, for $0\le i<n$;
so, in principle, we can determine the image of the generators of
$\sh$ under this isomorphism.

\begin{remark}(i)
Lusztig~\cite{L:iso} has shown that there exists a homomorphism $\Phi$
from the Hecke algebra $\sh(W)$ of any finite Weyl group $W$ to the
group ring $RW$ and he shows that $\Phi$ induces an isomorphism when
$\sh(W)$ is semisimple.  Ding and Hu \cite{Ding-Hu} have given a
 generalization of Lusztig's isomorphism theorem for cellular algebras, specially
 the cyclotomic Hecke algebras. Our map is not an analogue of Lusztig's
isomorphism.

 (ii) Brundan and Kleshchev \cite[Corollary 1.3]{BK-Block} showed that when $R$ is a field
  of characteristic zero, there is an isomorphism $\Psi$ between the cyclotomic Hecke algebras $H_{m,n}(q, Q)$ with $q$ not a root of unity and the degenerate cyclotomic Hecke algebras $\sh_{m,n}(Q)$. On the other hand, Lusztig \cite{Lusztig89} showed that there is a completion isomorphism $\Theta$ between the affine Hecke algebras $\sh^{\text{aff}}$  and the degenerate affine Hecke $\widehat{\sh}^{\text{aff}}$. Then the following diagram commutes
   $$\xymatrix{
  \sh^{\text{aff}}_n\ar[r]\ar[d]_{\Theta}&H_{m,n}(q, Q)\ar[d]_{\Psi}&H_q(S_n)\ar[l]\ar[d]_{\Phi}\\
  \widehat{\sh}_n^{\text{aff}}\ar[r]&\sh_{m,n}(Q)&RS_n,\ar[l]}$$
 where the other homomorphisms are the natural ones. Moreover, the isomorphism $\Psi$ can be given by the
   Mathas \cite[Corollary~2.14]{Mathas-J-Algebra} and Corollary~\ref{Cor: iso H G(m, 1,n)} when both $\sh_{m,n}(q, Q)$ and $\sh_{m, n}(Q)$ are semisimple over a field of characteristic zero.
\end{remark}

\begin{theorem}
Suppose that $R$ is a field and that $p_{\sh}(Q)\neq0$. Let $\gl$ be an $m$-multipartition of $n$.
\begin{enumerate}
\item Let $\ft$ be a standard $\lambda$-tableau. Then
$F_\ft=\frac1{\gamma_\ft}f_{\ft\ft}$ and $F_\ft$ is a primitive idempotent
with $S^\lambda\cong \sh F_\ft$.
\item Let $F_\lambda=\sum_{\ft\in\std(\lambda)} F_\ft$. Then
$F_\lambda$ is a primitive central idempotent.
\item $\{F_\lambda|\lambda\in\mpn\}$ is a complete set of
primitive central idempotents; in particular,
$$1=\sum_{\lambda\in\mpn}F_\lambda
   =\sum_{\ft\,\,\mathrm{standard}}F_\ft.$$
\end{enumerate}\label{Them idempotents}
\end{theorem}

\begin{proof}
Note that $F_{\ft}=\sum_{\fa,\fb}r_{\fa\fb}f_{\fa\fb}$ by Theorem~\ref{Them orthogonal basis}(ii).
Now by Proposition~\ref{Prop-f_st properties}(ii) and Theorem~\ref{Them orthogonal basis}(i)
$f_{\fs\ft}=f_{\fs\ft}F_{\ft}=\sum_{\fa,\fb}r_{\fa\fb}f_{\fs\ft}f_{\fa\fb}=\sum_{\fb}\gma_{\ft}r_{\ft\fb}f_{\fs\fb}$.
Equating coefficients, $r_{\ft\fb}=0$ if $\fb\neq \ft$ and $r_{\ft\ft}\gma_{\ft}=1$. Since $F_{\ft}^*=F_{\ft}$, we also
 have $r_{\fb\ft}=0$ if $\fb\neq\ft$. Hence $F_\ft=\frac{1}{\gma_{\ft}}$.
 By Theorem~\ref{Them orthogonal basis}(i), $F_\ft=\frac{1}{\gma_{\ft}}$ is an idempotent.  Further, $F_\ft$ is primitive
 since $S^{\gl}$ is irreducible and $S^{\gl}\cong \sh F_{\ft}$.

 (ii) and (iii) now follows because $\sh=\bigoplus_{\gl\in\mpn}\bigoplus_{\ft\in\std(\gl)}\sh F_{\ft}$ is a decomposition
 of $\sh$ into a direct sum of simple modules $\{S^{\gl}\cong\sh F_{\ft}\mid \gl\in\mpn, \ft\in\std(\gl)\}$.
\end{proof}

\begin{corollary}\label{Cor x_k F_k}Let $\ft$ be a standard tableau and let $k$ be an
integer with $1\le k\le n$. Then
\begin{enumerate}
\item $x_kF_\ft=\res_\ft(k)F_\ft$
and $x_kf_{\fs\ft}=\res_\fs(k)f_{\fs\ft}$.
\label{F_t L_k}

\item $\prod_{c\in\mcr(k)}(x_k-c)=0$ and
this is the minimum polynomial for $x_k$ acting on $\sh$.

\item $x_k=\sum_\ft \res_\ft(k) F_\ft$, where the sum is over the set of all
standard tableaux $($of arbitrary shape$)$.
\end{enumerate}
\end{corollary}
\begin{proof}(i) follows directly by Theorem~\ref{Them idempotents}(i) and Proposition~\ref{Prop-f_st properties}(iii).

(ii) First, $\prod_{c\in\mcr(k)}(x_k-c)=\sum_{\ft}\prod_{c\in\mcr(k)}(x_k-c)F_{\ft}=
\sum_{\ft}\prod_{c\in\mcr(k)}(\res_{\ft}(k)-c)=0$. Now if we remove any factor $x_k-c$ from the product
$\prod_{c\in\mcr(k)}(x_k-c)$. Then there exists a standard tableau $\ft$ such that $\res_{\ft}(k)=c'$ and for any standard
tableau $\fs\neq \ft$, $\res_{\fs}(k)\neq c'$. Hence $\prod_{\mcr(k)\ni c\neq c'}(x_k-c)=\sum_{\fs\neq\ft}\prod_{\mcr(k)\ni c\neq c'}(x_k-c)
F_{\ft}+\prod_{\mcr(k)\ni c\neq c'}(x_k-c)F_{\ft}\neq 0$, that is, $\prod_{c\in\mcr(k)}(x_k-c)$ is the minimal polynomial
for $x_k$ acting on $\sh$.

(iii) is follows directly by (i) and Theorem~\ref{Them idempotents}(iii). \end{proof}

Now we have a remark on the center of $\sh$ which described explicitly by
Brundan~\cite[Theorem~1]{B-repre08}.

\begin{remark}
The results in this section can be used to give a differential proof of the center of $\sh$, which
is the set of symmetric polynomials in $x_1, \dots, x_n$, when $\sh$ is semisimple.
\end{remark}

\section{Dual Specht modules}\label{sec: dual Spechts}

Let $\mathscr{Z}=\bz[\hat q_1,\dots,\hat q_m]$, where
$\hat{q}_1,\dots,\hat q_m$ are indeterminates over $\bz$,
and let $\sh_{\mathscr{Z}}$ be the degenerated cyclotomic Hecke algebra with parameters
$\hat{q}_1,\dots,\hat q_m$ over $\mathscr{Z}$. Consider the ring $R$ as a
$\mathscr{Z}$-module by letting $\hat q_i$ act on $R$ as multiplication by
$q_i$, for
$1\le i\le m$. Then $\sh\cong\sh_{\mathscr{Z}}\otimes_\mathscr{Z} R$, since $\sh$ is free as
an $R$-module; we say that $\sh$ is a {\it specialization} of $\sh_{\mathscr{Z}}$ and
call the map which sends $h\in\sh_{\mathscr{Z}}$ to $h\otimes1\in\sh$ the
{\it specialization homomorphism}.

\begin{definition}\label{Def bar-operation} Let $\bar{\,\,}: \mathscr{Z}\rr \mathscr{Z}$ be the $\bz$-linear map given by
$\hat q_i\mapsto -\hat q_{m-i+1}$, $x_i\mapsto -x_i$ for $1\le i\le m$, and $s_k=-s_k$ for
$1\le k<n$.\end{definition}

 Using the relations of $\sh_{\mathscr{Z}}$ it is easy to verify
that $\bar{\,\,}$ now extends to a $\bz$-linear ring involution $\bar{\,}: \sh_{\mathscr{Z}}\rr \sh_{\mathscr{Z}}$
of $\sh_{\mathscr{Z}}$. Hereafter, we drop the distinction between $\hat q_i$ and $q_i$.
Suppose that $h\in\sh$. Then there exists a (not necessarily unique)
$h_\mathscr{Z}\in\sh_{\mathscr{Z}}$ such that $h=h_\mathscr{Z}\otimes 1$ under specialization; we
sometimes abuse notation and write $\bar{h}=\overline{h_\mathscr{Z}}\otimes 1\in\sh$. As the
map $\bar{\,}$ does not in general define a semilinear involution on $R$,
this notation is not well-defined on elements of $\sh$; however, in
the cases where we employ it there should be no ambiguity. For
example, $\bar{w}=(-1)^{l(w)}w$  for all $w\in S_n$.

\begin{remark}The operation $\bar{\,\,}$ is very differential form the one defined by Mathas in \cite[\S3]{Mathas-J-Algebra} in the case of the cyclotomic Hecke algebras.\end{remark}

\begin{definition}\label{Def y_lamda u^-_lamda}Suppose that $\lambda=(\gl^1;\dots, \gl^m)$ is an $m$-multipartition of $n$.  Let
$y_\lambda=\sum_{w\in S_\lambda}(-1)^{l(w)}w$ and define
$n_\lambda=y_\lambda u_\lambda^-$ where
$$u_\lambda^-=(-1)^{n(\gl)}\prod_{i=2}^m\prod_{k=1}^{a_i}(x_k-q_{m-i+1})
             =(-1)^{n(\gl)}\prod_{i=1}^{m-1}\prod_{k=1}^{a_{m-i+1}}(x_k-q_{i}), $$
             where $a_i=|\gl^1|+\cdots+|\gl^{i-1}|$ for $i=1, \dots, m$ and $n(\gl)=\sum_{i=2}^ma_i=\sum_{i=1}^m(i-1)|\gl^i|$.

  For standard tableaux $\fs,\ft\in\std(\lambda)$ set
$n_{\fs\ft}=\overline{d(\fs)}^*n_\lambda \overline{d(\ft)}$.
\end{definition}

\begin{lemma}\label{Lem bar}Keep notations as above. Then $y_\lambda=\overline{x_\lambda}$, $u_\lambda^-=\overline{u_\lambda^+}$, and
$\overline{m_{\fs\ft}}\in\sh_{\mathscr{Z}}$ is mapped to $n_{\fs\ft}$ under specialization.
\end{lemma}
\begin{proof} All statements follows directly by definitions and the computations.\end{proof}

\begin{theorem}
The degenerate cyclotomic Hecke algebra $\sh$ is free as an $R$-module with
cellular basis
$\{n_{\fs\ft}|\fs,\ft\in\std(\lambda)\for \lambda\in\mpn\}$.
\end{theorem}
\begin{proof}The assertion follows by Theorem~\ref{std basis} and Lemma~\ref{Lem bar}. \end{proof}

Let $\lambda$ be an $m$-multipartition of $n$. Then $\overline{\sh^{\rhd\gl}}$ is a
two-sided ideal of $\sh$ which is free as an $R$-module with basis
$\{n_{\fa\fb}|\fa,\fb\in\std(\mu)\for \mu\rhd\lambda\}$. For simplicity, we write $\bar{\sh}^{\rhd\gl}$ for $\overline{\sh^{\rhd\gl}}$.

\begin{definition}
Let $\tilde S^\lambda$ be the Specht module corresponding to $\lambda$ determined by the basis $\{n_{\fs\ft}\}$, which is call the \textit{dual Specht module} corresponding to $\gl$.
\end{definition}

It is clearly that $\tilde S^\lambda\cong \sh n_\lambda/(\sh n_\lambda\cap\bar{\sh}^{\rhd\gl})$
and $\tilde S^\lambda$ is free as an $R$-module with basis
$\{n_\ft|\ft\in\std(\lambda)\}$, where
$n_\ft=n_{\ft\ft^{\gl}}+\bar{\sh}^{\rhd\gl}=\overline{m_{\ft\ft^\gl}+\sh^{\rhd\gl}}$ for all
$\ft\in\std(\lambda)$.

In order to compare the two modules $S^\lambda$ and $\tilde
S^\lambda$ we need to introduce some more notation.
Given a partition $\sigma$ let $\bar{\sigma}=(\bar{\sigma}_1,\bar{\sigma}_2,\dots)$
be the partition which is conjugate to $\sigma$; thus, $\bar{\sigma}_i$
is the number of nodes in column $i$ of the diagram of $\sigma$.
If $\lambda=(\lambda^1; \dots; \gl^m)$ is an $m$-multipartition then the {\it conjugate}
$m$-multipartition to $\gl$ is the $m$-multipartition
$\bar{\gl}=(\bar{\gl}^1, \dots, \bar{\gl}^m)$
with $\bar{\gl}^{i}=\overline{\lambda^{m-i+1}}$ for $1\le i\le m$.
Now suppose that $\ft=(\ft^1; \dots; \ft^m)$ is a standard $\lambda$-tableau.
Then the {\it conjugate} of $\ft$ is the standard $\bar{\gl}$-tableau
$\bar{\ft}=(\bar{\ft}^1; \dots; \bar{\ft}^m)$ where $\bar{\ft}^{i}$ is the tableau obtained by
interchanging the rows and columns of $\ft^{m-i+1}$.

\begin{example}\label{Example t^lamda t_lamda}
Let $\gl=(3\cdot1;4 \cdot 2)\in\mathscr{P}(2,10)$. Then $\bar{\gl}=(2 \cdot2 \cdot 1
 \cdot 1; 2 \cdot 1 \cdot 1)\in\mathscr{P}(2,10)$, and
$$\begin{array}{lll}\protect{
[\gl]=\left(\begin{array}{cc}\diagram{&&\cr \cr},\, &\diagram{&&&\cr &\cr}\end{array}\right)}
\,&\protect{
\ft^{\gl}=\left(\begin{array}{cc}\diagram{1&2&3\cr 4\cr},\, &\diagram{5&6&7&8\cr 9&10\cr}
\end{array}\right)}\,&
\protect{\overline{\ft^\gl}=\left(\begin{array}{cc}
\diagram{5&9\cr 6&10\cr 7\cr 8\cr},\,&\diagram{1&4\cr 2\cr 3\cr}
\end{array}\right)}\\
\protect{[\bar{\gl}]=\left(\begin{array}{cc}
\diagram{&\cr &\cr \cr \cr},\,&\diagram{&\cr \cr \cr}
\end{array}\right)}& \protect{\ft^{\bar{\gl}}=\left(\begin{array}{cc}
\diagram{1&2\cr 3&4\cr 5\cr 6\cr},\,&\diagram{7&8\cr 9\cr 10\cr}
\end{array}\right)}  & \protect{\overline{\ft^{\bar{\gl}}}=\left(\begin{array}{cc}
\diagram{7&9&10\cr 8\cr},\,&\diagram{1&3&5&6\cr 2&4\cr}\end{array}\right).}
\end{array}
$$
\end{example}

\begin{lemma}
 Let $\ft$ be a standard $\lambda$-tableau. Then
$\overline{\res_\ft(k)}=-\res_{\bar{\ft}}(k)$ in $\mathscr{Z}$, for $1\le k\le n$.
\label{res'}\end{lemma}
 \begin{proof}
 The lemma follows immediately form the definitions.
 \end{proof}

The point of Lemma~\ref{res'} is that the expression $\res_{\bar{\ft}}(k)$ is always well-defined; whereas
$\overline{\res_\ft(k)}$ is ambiguous for certain rings $R$. As a first
consequence we have the following fact.

\begin{proposition}
Let $\fs$ and $\ft$ be standard $\lambda$-tableaux and suppose
that $k$ is an integer with $1\le k\le n$. Then there exist
$r_\fb\in R$ such that
$$x_kn_{\fs\ft}=\res_{\bar{\fs}}(k)n_{\fs\ft}
       +\sum_{\std(\lambda)\ni\fb\rhd \fs}r_\fb n_{\fb\ft}
                \bmod\overline{\sh}^{\rhd\gl}.$$
\label{L_k' action}
\end{proposition}
\begin{proof}
First assume that $R=\mathscr{Z}$. Then $\bar{\,\,}$ is a $\bz$-linear ring
involution on $\sh_{\mathscr{Z}}$ and~$\bar{x}_k=-x_k$; therefore, by Theorem~\ref{x_k action},
$$\begin{aligned}&x_kn_{\fs\ft}=-\overline{x_km_{\fs\ft}}&=&
         -\overline{\Big(\res_{\fs}(k)m_{\fs\ft} +\sum_{\std(\lambda)\ni\fb\rhd \fs}r_\fb
         m_{\fb\ft}\bmod\sh^{\rhd\gl}\Big)}\\
       &&=&-\overline{\res_{\fs}(k)}\overline{m}_{\fs\ft} +\sum_{\std(\lambda)\ni\fb\rhd \fs}(-r_\fb)
       \overline{m}_{\fb\ft}\bmod\overline{\sh}^{\rhd\gl}\\
       &&=&\res_{\bar{\fs}}(k)n_{\fs\ft} +\sum_{\std(\lambda)\ni\fb\rhd \fs}(-r_\fb)
       n_{\fb\ft}\bmod\overline{\sh}^{\rhd\gl}.\end{aligned}$$
The general
case now follows by specialization since $\sh\cong\sh_{\mathscr{Z}}\otimes_\mathscr{Z} R$.
\end{proof}

Next consider the orthogonal basis $\{f_{\fs\ft}\}$ of $\sh$ in the
case where $P_\sh(Q)$ is invertible. Let $\mathscr{Z}_P$ be the localization
of $\mathscr{Z}$ at~$P_\sh(Q)$ and let $\sh_{\mathscr{Z}_P}$ be the corresponding
degenerate cyclotomic Hecke algebra.  The involution $\bar{\,}$ extends to $\sh_{\mathscr{Z}_P}$
and $\sh$ is a specialization of $\sh_{\mathscr{Z}_P}$ whenever $P_\sh(Q)$ is
invertible in $R$. (Note that $q_1,\dots,q_m$ are indeterminates in
$\mathscr{Z}_P$.)

In general, $f_{\fs\ft}\notin\sh_{\mathscr{Z}}$; however, if
$\ft\ne\fu$ then $\res_\ft(k)-\res_\fu(k)$ is a factor of $P_\sh(Q)$ for
all $k$, so $f_{\fs\ft}\in\sh_{\mathscr{Z}_P}$ and we can speak of the elements
$F_\ft$ and $f_{\fs\ft}\in\sh_{\mathscr{Z}_P}$. More generally, whenever
$P_\sh(Q)$ is invertible in $R$ we have an element $\bar{f}_{\fs\ft}\in\sh$
via specialization because $\sh\cong\sh_{\mathscr{Z}_P}\otimes_{\mathscr{Z}_P}R$.

\begin{proposition}
Suppose that $\ft$ is a standard tableau. Then $\overline{F_\ft}=F_{\bar{\ft}}$
in $\sh_{\mathscr{Z}_P}$.
\label{F'}
\end{proposition}

\begin{proof}Note that $\overline{\mcr(k)}=\mcr(k)$ and $\res_{\fu}(k)\neq \res_{\ft}(k)$ if and only if
$\res_{\bar{\fu}}(k)\neq \res_{\bar{\ft}}(k)$. Now
applying the definitions together with Lemma~\ref{res'} gives
\begin{align*}
\overline{F_\ft}&=\overline{\prod_{k=1}^n
   \prod_{\substack{c\in\mcr(k)\\c\ne\res_\ft(k)}}
\frac{x_k-c}{\res_\ft(k)-c}}\\
       &=\overline{\prod_{k=1}^n\prod_{\substack{\mu\in\mpn,\fu\in\std(\mu)\\\res_{\fu}(k)\ne\res_{\ft}(k)}}
  \frac{x_k-\res_{\fu}(k)}{\res_{\ft}(k)-\res_{\fu}(k)}}\\
  &=\prod_{k=1}^n\prod_{\substack{\mu\in\mpn,\fu\in\std(\mu)\\\res_{\fu}(k)\ne\res_{\ft}(k)}}
  \frac{\overline{x_k}-\overline{\res_{\fu}(k)}}{\overline{\res_{\ft}(k)}-\overline{\res_{\fu}(k)}}\\
  &=\prod_{k=1}^n\prod_{\substack{\mu\in\mpn,\fu\in\std(\mu)\\\res_{\bar{\fu}}(k)\ne\res_{\bar{\ft}}(k)}}
  \frac{x_k-\res_{\bar{\fu}}(k)}{\res_{\bar{\ft}}(k)-\res_{\overline{\fu}}(k)}\\
      &=\prod_{k=1}^n\prod_{\substack{c\in\mcr(k)\\c\ne\res_{\bar{\ft}}(k)}}
  \frac{x_k-c}{\res_{\bar{\ft}}(k)-c}
       =F_{\bar{\ft}}.
\end{align*}
\end{proof}

Now we let $g_{\fs\ft}=F_{\bar{\fs}}n_{\fs\ft}F_{\bar{\ft}}$;
then in $\sh_{\mathscr{Z}_p}$, $g_{\fs\ft}=\overline{F}_\fs \overline{m}_{\fs\ft}\overline{F}_\ft=\overline{f}_{\fs\ft}$.
Applying $\bar{\,}$ to $\{f_{\fs\ft}\}$ and using Theorem~\ref{Them orthogonal basis}(ii) (and a
specialization argument) shows that
$\{g_{\fs\ft}|\fs,\ft\in\std(\lambda), \for \lambda\in\mpn\}$ is a
basis of $\sh$. Consequently, as in Corollary~\ref{S=fH},
$\tilde S^\lambda\cong
\sh g_{\fs\ft}$ for any standard $\lambda$-tableaux
$\fs,\ft\in\std(\lambda)$.

\begin{remark}
By the Proposition and Theorem~\ref{Them idempotents}(i),
$$g_{\ft\ft}=\overline{f_{\ft\ft}}=\overline{\gamma_\ft F_\ft}=\overline{\gamma_\ft} F_{\bar{\ft}}
         =\frac {\overline{\gamma_\ft}}{\gamma_{\bar{\ft}}}f_{\bar{\ft}\bar{\ft}}.$$
More generally, we can write $g_{\fs\ft}=\sum_{\fa,\fb}r_{\fa\fb}f_{\fa\fb}$
for some $r_{\fa\fb}\in R$. By Propositions~\ref{Prop-f_st properties} and \ref{F'},
$F_{\bar{\fs}}g_{\fs\ft}F_{\bar{\ft}}=g_{\fs\ft}$; so it follows that $r_{\fa\fb}=0$
unless $\fa=\bar{\fs}$ and $\fb=\bar{\ft}$. Therefore,
$g_{\fs\ft}=\alpha_{\fs\ft}f_{\bar{\fs}\bar{\ft}}$ for some $\alpha_{\fs\ft}\in R$.
Applying the $*$-involution shows that $\alpha_{\fs\ft}=\alpha_{\ft\fs}$.
Finally, by looking at the product $g_{\fs\ft}g_{\ft\fs}$ we see that
$\alpha_{\fs\ft}^2=\overline{\gamma_\fs}\overline{\gamma_\ft}/\gamma_{\bar{\fs}}\gamma_{\bar{\ft}}$.
\label{g_st to f_st}\end{remark}

Combining Proposition~\ref{F'} with Corollary~\ref{S=fH} and the corresponding result for the
$g$-basis shows that
$ S^\lambda\cong \sh f_{\ft\ft}=\sh g_{\bar{\ft}\bar{\ft}}\cong\tilde S^{\bar{\gl}}$, for any
$\ft\in\std(\lambda)$. Hence, we have the following.

\begin{corollary}
Assume that  Assumption~\ref{ppoly} holds. Then
$\tilde S^\lambda\cong S^{\bar{\gl}}$.
\label{dual isomorphism}\end{corollary}

\begin{remark}
When $R$ is field the assumption that Assumption~\ref{ppoly} holds is
equivalent to $\sh$ being semisimple. This assumption is necessary
because, in general, $S^{\bar{\gl}}$ and $\tilde S^\lambda$ are not
isomorphic; rather, we can show that
$S^{\bar{\gl}}$ is isomorphic to the {\it dual} of $\tilde S^\lambda$, the detail will be appear elsewhere.
 In the semisimple case both
$S^\lambda$ and $\tilde S^{\bar{\gl}}$ are irreducible, and hence
self-dual, since they carry a non-degenerate bilinear form.
Accordingly, we call the module $\tilde S^\lambda$ the dual Specht
module.\end{remark}

\begin{corollary} Let $\lambda$ and $\mu$ be $m$-multipartitions of $n$.
Suppose that $\fs$ and $\ft$ are standard $\lambda$-tableaux and that
$\fa$ and $\fb$ are standard $\mu$-tableaux. Then $f_{\fs\ft}g_{\fa\fb}=\gd_{\ft\bar{\fa}}r_{\ft\fb}f_{\fs\bar{\fb}}$
for some $r_{\ft\fb}\in R$.
\label{cancellation}
\end{corollary}

\begin{proof}
Applying the definitions and Remark~\ref{g_st to f_st},
$$f_{\fs\ft}g_{\fa\fb}
   =F_\fs m_{\fs\ft}F_\ft \bar{F}_\fa \bar{m}_{\fa\fb}\bar{F}_\fb=F_\fs m_{\fs\ft}
   F_\ft F_{\bar{\fa}} \overline{m}_{\fa\fb}\bar{F}_\fb=\delta_{\ft\bar{\fa}}F_\fs
   m_{\fs\ft}F_\ft\overline{m}_{\fa\fb}\bar{F}_\fb=\gd_{\ft\bar{\fa}}r_{\ft\fb}f_{\fs\bar{\fb}}$$
  for some $r_{\ft\fb}\in R$.  It completes the proof. \end{proof}

The Specht modules $S^\lambda$ and the dual Specht modules $\tilde
S^\lambda$ are both constructed as quotient modules using the cellular bases
$\{m_{\fs\ft}\}$ and $\{n_{\fs\ft}\}$ respectively (see Corollary~\ref{S=fH}).
Using the orthogonal basis $\{f_{\fs\ft}\}$ and $\{g_{\fs\ft}\}$
we have also constructed these modules as submodules of $\sh$.

\vspace{2\jot}
Recall that $\ft^{\gl}$ is the $\lambda$-tableau which has the numbers
$1,2,\dots,n$ entered in order first along the rows of $\ft^{\gl^1}$
and then the rows of $\ft^{\gl^2}$ and so on. Let
$\ft_{\gl}=\overline{\ft^{\bar{\gl}}}$, that is,  $\ft_{\gl}$ is the
$\lambda$-tableau with the numbers $1,2,\dots,n$ entered in order
first down the columns of $\ft_{\gl^m}$ and then the columns of
$\ft_{\gl^{m-1}}$ etc,  see Example~\ref{Example t^lamda t_lamda}. Observe that if $\ft$ is a standard
$\lambda$-tableau then $\ft^\gl\unrhd\ft\unrhd\ft_{\gl}$.

\begin{proposition}
Suppose that $\lambda$ be an
$m$-multipartition of $n$. Then $m_\lambda\sh  n_{\bar{\gl}}=Rf_{\ft^{\gl}\ft_{\gl}}$.
\label{Prop m_lamda n_bar lamda}
\end{proposition}

\begin{proof}
By Proposition~\ref{Prop-f_st properties}(i),
$m_\lambda=f_{\ft^{\gl}\ft^{\gl}}+\sum_{\fa,\fb\rhd\ft^{\gl}}r_{\fa\fb}f_{\fa\fb}$ for
some $r_{\fa\fb}\in R$. By interchanging the roles of $\lambda$
and $\bar{\gl}$ and by applying the involution $\bar{\,}$ (in $\sh_{\mathscr{Z}_p}$ and
then specializing) we see that there exist $r_{\ga\gb}\in R$ such that
$$(\dag)\quad\quad n_{\bar{\gl}}=g_{\ft^{\bar{\gl}}\ft^{\bar{\gl}}}
             +\sum_{\ga,\gb\rhd\ft^{\bar{\gl}}}r_{\ga\gb}g_{\ga\gb}
      =\frac{\overline{\gamma_{\ft^{\bar{\gl}}}}}{\gamma_{\ft_{\gl}}}f_{\ft_{\gl}\ft_{\gl}}
             +\sum_{\ft_{\gl}\rhd\bar{\ga},\bar{\gb}}r_{\ga\gb}g_{\ga\gb}
$$
where for the second equality we have used Remark~\ref{g_st to f_st} (note
that $\overline{\ft^{\bar{\gl}}}=\ft_{\gl}$) and the observation that
$\ga,\gb\rhd\ft^{\bar{\gl}}$ if and only if $\ft_{\gl}\rhd\bar{\ga},\bar{\gb}$.  Now
$m_\lambda\sh  n_{\bar{\gl}}$ is spanned by the elements $m_\lambda f_{\fs\ft}
n_{\bar{\gl}}$, where $\fs$ and $\ft$ range over all pairs of standard
tableaux of the same shape. Now, $(\dag)$ and
Corollary~\ref{cancellation} imply that
\begin{align*}
m_{\lambda} f_{\fs \ft} n_{\bar{\gl}}
&=\biggl(f_{{\ft^\gl}{\ft^\gl}}+\sum_{\fa,\fb\rhd\ft^{\gl}}r_{\fa\fb}f_{\fa\fb}\biggr)f_{\fs\ft}
\biggl(\frac{\overline{\gamma_{\ft^{\bar{\gl}}}}}{\gamma_{\ft_{\gl}}}f_{\ft_{\gl}\ft_{\gl}}+
\sum_{\ft_{\gl}\rhd\bar{\ga},\bar{\gb}}r_{\ga\gb}g_{\ga\gb}\biggr)\\
&=\frac{\overline{\gamma_{\ft^{\bar{\gl}}}}}{\gamma_{\ft_{\gl}}}
f_{\ft^{\gl}\ft^{\gl}}f_{\fs\ft}f_{\ft_{\gl}\ft_{\gl}}
        +\sum_{\substack{\ft_{\gl}\rhd\bar{\ga},\bar{\gb}\\
        \fa,\fb\rhd\ft^{\gl}}}r_{\fa\fb}r_{\ga\gb}f_{\fa\fb}f_{\fs\ft}g_{\ga\gb},\\
&=\frac{\overline{\gamma_{\ft^{\bar{\gl}}}}}{\gamma_{\ft_{\gl}}}f_{\ft^{\gl}\ft^{\gl}}
f_{\fs\ft}f_{\ft_{\gl}\ft_{\gl}}\\
&=\begin{cases}\gamma_{\ft^{\gl}}\overline{\gamma_{\ft^{\bar{\gl}}}}f_{\ft^{\gl}\ft_{\gl}},
&\If\fs=\ft^{\gl}\And\ft=\ft_{\gl},\\
	             0,&\otherwise,
\end{cases}
\end{align*}
With the third equality following the facts that $f_{\fa\fb}f_{\fs\ft}
g_{\ga\gb}=\gd_{\fb\fs}\gd_{\ft\bar{\fa}}rf_{\fa\bar{\gb}}$ for some $r\in R$ and that
$\ft_{\gl}\rhd \bar{\ga}=\ft\rhd \ft_{\gl}$ is impossible.
\end{proof}

Let $w_\lambda=d(\ft_{\gl})$. Then $w_\lambda$ is the unique element of
$S_n$ such that $\ft_{\gl}=\ft^{\gl} w_\lambda$.

\begin{definition}
Suppose that $\lambda$ is an $m$-multipartition of $n$. Let
$z_{\lambda}=m_{\lambda} w_\lambda n_{\bar{\gl}}$.
\label{Def Z_lambda}\end{definition}
The element $z_\lambda$ and the following result are crucial to our
computation of the Schur elements.

\begin{proposition}
Assume that Assumption~\ref{ppoly} holds. Then
$z_\lambda\!\!=\!\!\overline{\gamma_{\ft^{\bar{\gl}}}}f_{\ft^{\gl}\ft_{\gl}}$ and
$m_\lambda\sh n_{\bar{\gl}}\!\!=\!Rz_\lambda$.
\label{z=f}
\end{proposition}
\begin{proof}By definition and applying the equality~$(\dag)$ in the proof of Proposition~\ref{Prop m_lamda n_bar lamda}, \begin{align*}z_{\lambda}&= m_{\lambda} w_{\lambda} n_{\bar{\gl}}=m_{\ft^{\gl}
\ft_{\gl}}n_{\bar{\gl}}\\
&=\biggl(f_{\ft^{\gl}\ft_{\gl}}+\sum_{\fa,\fb\rhd\ft_{\gl}} r_{\fa\fb}f_{\fa\fb}\biggr)
\biggl(\frac{\overline{\gamma_{\ft^{\bar{\gl}}}}}{\gamma_{\ft_{\gl}}}f_{\ft_{\gl}\ft_{\gl}}+
\sum_{\ft_{\gl}\rhd\bar{\ga},\bar{\gb}}r_{\ga\gb}g_{\ga\gb}\biggr)\\ &=\frac{\overline{
\gamma_{\ft^{\bar{\gl}}}}}{\gamma_{\ft_{\gl}}}f_{\ft^{\gl}\ft_{\gl}}f_{\ft_{\gl}\ft_{\gl}}\\
&=\overline{\gamma_{\ft^{\bar{\gl}}}}f_{\ft^{\gl}\ft_{\gl}},
\end{align*}
where the last equality follows by Theorem~\ref{Them orthogonal basis}(i).
\end{proof}	

Now, $z_\lambda$ is an element of $\sh_\mathscr{Z}$, so
$\overline{\gamma_{\ft^{\bar{\gl}}}}f_{\ft^{\gl}\ft_{\gl}}\in\sh_\mathscr{Z}$.
By definition, $\sh z_\lambda$
is a quotient module of $\sh m_\lambda$ and a submodule of $\sh n_{\bar{\gl}}$.

\begin{remark}
Over an arbitrary ring $R$, along the line of Du and
Rui~\cite[Remark~2.5]{DuRui:branching}, we can  show that $S^\lambda\cong
\sh z_\lambda^*$ and $\tilde S^{\bar{\lambda}}\cong \sh z_\lambda$ as
$\sh$-modules, the isomorphisms being given by the natural quotient
maps $\sh m_\lambda \rr \sh z_\lambda^*$ and $\sh n_{\bar{\gl}}\rr \sh z_\lambda$.
Note that $S^\lambda\cong\tilde S^{\bar{\gl}}$ when $\sh$ is semisimple by
Corollary~\ref{S=fH}.\end{remark}

\section{Some nice primitive idempotents}\label{Sec: nice primitive idempotents}
In this section we give a simple formula for the primitive idempotents $F_{\ft^{\gl}}=
\frac{1}{\gma_{\ft^{\gl}}}f_{\ft^{\gl}}$.

\begin{proposition}
Suppose that $\ft$ is a standard $\lambda$-tableau. Then
there exist invertible elements $\Phi_\ft$ and $\Psi_\ft$ in the group algebra $RS_n$ such
that
\begin{enumerate}
\item $\Psi_\ft F_\ft=F_{\ft^{\gl}}\Phi_\ft$;
\item $\Phi_\ft=d(\ft)+\sum_{w<d(\ft)} r_{\ft w}w$, for some
$r_{\ft w}\in R$; and
\item $\gamma_{\ft^\gl}\Phi_\ft \Psi_\ft^*=\gamma_\ft$.
\end{enumerate}\label{Phi-Psi prop}
\end{proposition}

\begin{proof}
We prove all three statements by induction on $\ft$ with respect to the dominance $\unrhd$. When
$\ft=\ft^\gl$ there is nothing to prove as we may take
$\Phi_{\ft^\gl}=\Psi_{\ft^\gl}=1$. Suppose then that $\ft\neq \ft^\gl$. Then there
exists an integer $i$, with $1\le i<n$, such that
$\fs=\ft(i,i+1)\rhd\ft$, or equivalently $\ft=\fs(i,i+1)\rhd \fs$.

Let
$\alpha=\frac{1}{\res_\ft(i)-\res_\fs(i)}$ and
$\beta=-\frac{1}{\res_\fs(i)-\res_\ft(i)}=-\ga$. Then $\ga+\gb=0$, $1+\ga\gb=\frac{\gma_{\ft}}{\gma_{\fs}}$, and $(s_i-\beta)f_{\ft\ft}=(\gamma_\ft/\gamma_\fs)f_{\fs\ft}$ according to Theorems~\ref{gamma properties}(i.b) and \ref{Prop-s_i f_su}. Similarly, by the right hand analogue of
Theorem~\ref{Tf multiplication}  (interchanging the roles of $\fs$ and $\ft$), $f_{\fs\fs}(s_i-\alpha)=f_{\fs\ft}$. Therefore,
\begin{align*}(\ddag)\quad\quad (s_i-\beta)F_\ft=\frac1{\gamma_\ft}(s_i-\beta)f_{\ft\ft}
                 =\frac1{\gamma_\fs}f_{\fs\ft}
		 =\frac1{\gamma_\fs}f_{\fs\fs}(s_i-\alpha)
		 =F_\fs(s_i-\alpha).\end{align*}
By induction, there exist invertible elements $\Phi_\fs$ and $\Psi_\fs$ in $RS_n$ satisfying properties (i)--(iii).

 Now define $\Psi_\ft=\Psi_\fs(s_i-\beta)$
and $\Phi_\ft=\Phi_\fs(s_i-\alpha)$. Then, by induction and the
equation $(\ddag)$,
$$ \Psi_\ft F_\ft=\Psi_\fs(s_i-\beta)F_\ft
                  =\Psi_\fs F_\fs(s_i-\alpha)
		  = F_{\ft^\gl}\Phi_\fs(s_i-\alpha)
		  = F_{\ft^\gl}\Phi_\ft.$$
Hence, (i) holds. Next, again by induction,
$$\Psi_\ft=\Psi_\fs(s_i-\beta)
         =\Big(d(\fs)+\sum_{v<d(\fs)} p_{\fs v}v\Big)(s_i-\beta)
         =d(\ft)+\sum_{w<d(\ft)} p_{\ft w}w,$$
by the standard properties of the Bruhat order since
$d(\ft)=d(\fs)s_i>d(\fs)$ according to Lemma~\ref{Lem Ehresmann}. Furthermore, by induction, $\Psi_{\ft}$ and
$\Phi_{\ft}$ are invertible because $s_i-\ga$ and $s_i-\gb$ are invertible in $RS_{n}$. (ii) is proved.

Finally, using
induction once more (and a quick calculation for the second
equality),
\begin{align*}
\gamma_{\ft^\gl}\Phi_\ft \Psi_\ft^*
    &=\gamma_{\ft^\gl}\Phi_\fs(s_i-\alpha)(s_i-\beta)\Psi_\fs^* \\
    &=\gamma_{\ft^\gl}\Phi_\fs\biggl(1+\alpha\beta-(\ga+\gb)s_i\biggr)\Psi_\fs^*\\
    &= \gamma_\ft,
\end{align*}
This proves (iii) and so completes the proof.
\end{proof}

\begin{remark}
We are not claiming that the elements $\Phi_\ft$ and $\Psi_\ft$ are
uniquely determined by the conditions of the Proposition; ostensibly,
these elements depend upon the choice of reduced expression for $d(\ft)$.
 Indeed, if $s_{i_1}\dots s_{i_l}$ is a reduced expression of $d(\ft)$,
then we may choose that
\begin{align*}\Psi_{\ft}=(s_{i_1}-\gb_{i_1})\dots (s_{i_l}-\gb_{i_l}) \quad\quad\text{ and }\quad\quad
\Phi_{\ft}=(s_{i_1}-\ga_{i_1})\dots (s_{i_l}-\ga_{i_l}),\end{align*}
 where $\ft_j=\ft(s_{i_1}\cdots s_{i_j})$, $\ga_{i_j}=\frac{1}{\res_{\ft_j}(i_j)-\res_{\ft_{j-1}}(i_j)}$, and $\gb_{i_j}=-\ga_{i_j}$.\end{remark}

\begin{corollary}
Suppose that $\ft$ is a standard $\lambda$-tableau.  Then
\begin{enumerate}
\item $F_{\ft^\gl}=\frac{\gamma_{\ft^{\gl}}}{\gamma_\ft}\Psi_\ft F_\ft\Psi_\ft^*$; and,
\item $F_\ft=\frac{\gamma_{\ft^{\gl}}}{\gamma_\ft}\Phi_\ft^* F_{\ft^{\gl}}\Phi_\ft$.
\end{enumerate}\label{Cor F_lam Psi Phi}
\end{corollary}
\begin{proof}
Using parts (iii) and (i) of the Proposition, shows that
$$F_{\ft^\gl}=(\frac{\gamma_{\ft^{\gl}}}{\gamma_\ft}\Psi_\ft\Phi_\ft)^*F_{\ft^{\gl}}
         =\frac{\gamma_{\ft^{\gl}}}{\gamma_\ft}\Psi_\ft F_\ft\Psi_\ft^* \quad\quad\text{ and }\qquad
         F_{\ft}=F_{\ft}(\frac{\gamma_{\ft^{\gl}}}{\gamma_\ft}\Psi_\ft^*\Phi_\ft)=
         \frac{\gamma_{\ft^{\gl}}}{\gamma_\ft}\Phi_\ft^* F_\ft\Phi_\ft,$$
where we use the  `conjugating' version of Proposition~\ref{Phi-Psi prop} for the proof of the part (ii).
\end{proof}

The main reason why we are interested in $\Psi_\ft$ and $\Phi_\ft$
is the following.

\begin{lemma}
Suppose that $\fs$ and $\ft$ are standard
$\lambda$-tableaux. Then
$$f_{\fs\ft}=\Psi_\fs^* f_{\ft^\gl\ft_\gl}\Psi_\ft.$$
\end{lemma}

\begin{proof}
By the definition of $f_{\fs\ft}$ and Proposition~\ref{Phi-Psi prop}(ii)
we have
$$f_{\fs\ft}=F_\fs d(\fs)^*m_\lambda d(\ft)  F_\ft
	  =F_\fs \Phi_\fs^* m_\lambda \Phi_\ft F_\ft
	      -\sum_{(v,w)<(d(\fs),d(\ft))}
	      p_{\fs v}p_{\ft w}F_\fs v^*m_\lambda w F_\ft.$$
Now if $(v,w)<(d(\fs),d(\ft))$ then, by Lemma~\ref{Lem Ehresmann}, $v^*m_\lambda w$ belongs
to the span of the $m_{\fa\fb}$ where $(\fa,\fb)\rhd(\fs,\ft)$.
Therefore, by Proposition~\ref{Prop-f_st properties}(i), $v^*m_\lambda w$
belongs to the span of the $f_{\fa\fb}$ where either $\fa$ and $\fb$
are standard $\lambda$-tableaux and $(\fa,\fb)\rhd(\fs,\ft)$,
or $\shape(\fa)=\shape(\fb)\rhd\lambda$; consequently, $F_\fs
v^*m_\lambda w F_\ft=0$ by Proposition~\ref{Prop-f_st properties}(iii). Hence,
by Theorem~\ref{Them idempotents}(i) and Proposition~\ref{Phi-Psi prop}(i),
$$f_{\fs\ft} =F_\fs \Phi_\fs^* m_\lambda \Phi_\ft F_\ft
           =\Psi_\fs^* F_{\ft^\gl} m_\lambda F_{\ft^\gl}\Psi_\ft
           =\Psi_\fs^* f_{\ft^\gl\ft_\gl}\Psi_\ft^*$$
as required.
\end{proof}

Recall that, see Proposition~\ref{Prop m_lamda n_bar lamda}, $z_{\gl}=m_{\gl}w_{\gl}n_{\bar{\gl}}=
u^{+}_{\gl}x_{\gl}w_{\gl}y_{\bar{\gl}}
u^-_{\bar{\gl}}=\overline{\gamma_{\ft^{\bar{\gl}}}}f_{\ft^{\gl}\ft^{\gl}}$. The following fact is crucial to our
computation of the Schur elements of the degenerate cyclotomic Hecke algebras.

\begin{proposition}\label{Prop F_t^lambda Z_lambda}
Suppose that $\lambda$ is an $m$-multipartition of $n$. Then
\begin{align*}F_{\ft^{\gl}}=\frac1{\gamma_{\ft_{\gl}}} f_{\ft^{\gl}\ft_{\gl}}\Phi_{\ft_{\gl}}
  =\frac1{\gamma_{\ft_{\gl}}\overline{\gamma_{\ft^{\bar{\gl}}}}}z_\lambda w_{\bar{\gl}}\end{align*}
  is a primitive idempotent of $\sh$ with $\sh F_{\ft^{\gl}}\cong S^{\gl}$.
\end{proposition}
\begin{proof}By Corollary~\ref{Cor F_lam Psi Phi}(i) and Proposition~\ref{Phi-Psi prop}(iii),
$F_{\ft^\gl}=\frac{\gamma_{\ft^{\gl}}}
{\gamma_\ft}\Psi_\ft F_\ft\Psi_\ft^*=\gma_{\ft^{\gl}}F_{\ft^{\gl}}\Phi_{\ft_{\gl}}$. Now, by Proposition~\ref{z=f},
$z_{\gl}=
\overline{\gamma_{\ft^{\bar{\gl}}}}f_{\ft^{\gl}\ft_{\gl}}=\gma_{\ft^{\gl}}\overline{
\gamma_{\ft^{\bar{\gl}}}}F_{\ft^{\gl}}\Phi_{\ft_{\gl}}$.
 Therefore $F_{\ft^{\gl}}=\frac{1}
{\overline{\gamma_{\ft^{\bar{\gl}}}}\gma_{\ft_{\gl}}}z_{\gl}\Phi_{\ft_{\gl}}$, which implies that $z_{\gl}\in F_{\ft^{\gl}}\sh$, moreover
$F_{\ft^{\gl}}z_{\gl}=z_{\gl}$ since $F_{\ft^{\gl}}$ is an idempotent. As a consequence,
\begin{align*}z_{\gl}w_{\gl}=F_{\ft^{\gl}}z_{\gl}w_{\gl}=\frac{1}
{\overline{\gamma_{\ft^{\bar{\gl}}}}\gma_{\ft_{\gl}}}z_{\gl}\Phi_{\ft_{\gl}}z_{\gl}w_{\gl}=\frac{1}
{\overline{\gamma_{\ft^{\bar{\gl}}}}\gma_{\ft_{\gl}}}m_{\gl}w_{\gl}n_{\bar{\gl}}\Phi_{\ft_{\gl}}
m_{\gl}w_{\gl}n_{\bar{\gl}}w_{\gl}.\end{align*}

Applying Proposition~\ref{Phi-Psi prop}(ii), $\Phi_{\ft_{\gl}}=w_{\gl}+\sum_{w<w_{\gl}} r_{\ft_{\gl} w}w$, for some
$r_{\ft_{\gl} w}\in R$. On the other hand, note that the permutation $w_{\gl}$ has the ``trivial intersection property'',
that is, $S_{\gl}\cap wS_{\bar{\gl}}w^{-1}\neq \{1\}$ if and only if $S_{\gl}wS_{\bar{\gl}}=S_{\gl}w_{\gl}S_{\bar{\gl}}$,
see for example \cite[\S4.9]{DJ:reps}.
As a consequence, $y_{\bar{\gl}}wx_{\gl}\neq 0$ if and only if $w\in S_{\gl}w_{\gl}S_{\bar{\gl}}$. Thus
$y_{\bar{\gl}}\Phi_{\ft_{\gl}}x_{\gl}=y_{\bar{\gl}}w_{\gl}x_{\gl}$ since $w_{{\gl}}$ is the unique element of minimal
 length in $S_{\gl}w_{\gl}S_{\bar{\gl}}$, and $z_{\gl}\Phi_{\ft_{\gl}}z_{\gl}=(x_{\gl}u^+_{\gl}w_{\gl}
 u^-_{\bar{\gl}}y_{\bar{\gl}})\Phi_{\ft_{\gl}}(x_{\gl}u^+_{\gl}w_{\gl}
 u^-_{\bar{\gl}}y_{\bar{\gl}})=x_{\gl}u^+_{\gl}w_{\gl}
 u^-_{\bar{\gl}}y_{\bar{\gl}}w_{\bar{\gl}}x_{\gl}u^+_{\gl}w_{\gl}
 u^-_{\bar{\gl}}y_{\bar{\gl}})=z_{\gl}w_{\bar{\gl}}z_{\gl}$.  Therefore,  $z_{\gl}w_{\bar{\gl}}=\frac{1}
{\overline{\gamma_{\ft^{\bar{\gl}}}}\gma_{\ft_{\gl}}}(z_{\gl}w_{\gl})^2$. Furthermore, $\frac{1}
{\overline{\gamma_{\ft^{\bar{\gl}}}}\gma_{\ft_{\gl}}}z_{\gl}w_{\gl}$ is an idempotent in $\sh F_{\ft^{\gl}}$. Hence, $F_{\ft^{\gl}}=\frac{1}
{\overline{\gamma_{\ft^{\bar{\gl}}}}\gma_{\ft_{\gl}}}z_{\gl}w_{\gl}$ since $F_{\ft^{\gl}}$ is a primitive idempotent.
\end{proof}
\section{Computation of $\tau(z_\gl w_{\bar{\gl}})$}\label{Sec: comuputaion}
In this section we determine $\tau(z_\gl w_{\bar{\gl}})$, which is the key to the computation of the
Schur elements of the degenerate cyclotomic Hecke algebras $\sh$. We will see that for all parameters $q_1, \dots, q_m$,
$\tau(z_\gl w_{\bar{\gl}})$ is a unit in $R$, which answer partly that for all parameters $q_1, \dots, q_m$,
$\sh$ is a symmetric algebra. This fact has important consequences for the representation theory of degenerate Hecke
algebras. For example, it can be used to show that $S^{\gl}=\sh m_{\gl}$ is a self dual $\sh$-module and that
$\widetilde{S}^{\bar{\gl}}$ is isomorphic to the dual of $S^{\gl}$. The details will be appear elsewhere.

\begin{Point}{}*
Fix an $m$-multipartition $\gl=(\gl^1; \dots; \gl^m)$ of $n$. Let $a_i=\sum_{j=1}^{i-1}|\lambda^{j}|$ and
$b_i=\sum_{j=i+1}^m|\lambda^{j}|$ for $1\le i\le m$. Set
 $n(\gl)=\sum_{i=1}^m=(i-1)|\gl^i|$ and $\ga(\gl)=\frac{1}{2}\sum_{i=1}^m\sum_{j\geq1}(\gl^i_j-1)\gl^i_j$.  For each integer $k$ with $1\le k\le n$,  we let $c^{\gl}(k)=c$ (resp., $c_{\gl}(k)$) if $k$ appears in the $c$-component of $\ft^{\gl}$ (resp., $\ft_{\gl}$). Then $c^{\gl}(k)=
\min\{1\le c\le m\mid k\leq |\gl^1|+\dots +|\gl^c|\}$ and $c_{\gl}(k)=
\min\{1\le c\le m\mid k\leq |\gl^{m+1-c}|+\dots +|\gl^m|\}$.

Recall that $u^+_{\gl}=\prod_{i=2}^m\prod_{k=1}^{a_i}(x_k-q_i)$ and $u^-_{\bar{\gl}}
=(-1)^{n(\gl)}\prod_{j=2}^m\prod_{l=1}^{b_j}(x_k-q_l)$ (Definitions~\ref{Def x_lamda u^+_lamda} and
\ref{Def y_lamda u^-_lamda}). Clearly, $u^+_{\gl}$ is a polynomial (in
variables $x_1, \dots, x_n$) of degree $\ga(\gl)$ and $u^-_{\bar{\gl}}$ is a polynomial (in
variables $x_1, \dots, x_n$) of degree $n(\gl)$.  Furthermore, for each integer $k$ with $1\le k\le n$, $u^+_{\gl}$
is a polynomial in $x_k$ of degree $m-c_{\gl}(k)$.

Recall that there is a natural $S_n$-action on the polynomial ring $R[y_1, \dots, y_n]$ defined by $$\sigma\cdot f(y_1, \dots,
y_n)=f(y_{\sigma(1)}, \dots, y_{\sigma(n)})$$ for all $\sigma\in S_n$ and for all
$f(y_1, \dots, y_n)\in R[y_1, \dots, y_n]$. For a moment, we denote by $L(f)$ the leading term of a polynomial
$f(y_1, \dots, y_n)\in R[y_1, \dots, y_n]$.
\end{Point}

Before we state our key lemma, we consider an example.

\begin{example}Let $\gl=(2\cdot1\cdot1;2\cdot1;1)$ be a 3-multipartition of 8. Then $\bar{\gl}=(1; 2\cdot1;3\cdot1)$,
$a_1=0=b_3$, $a_2=4$, $a_3=7$, $b_1=4$, $b_2=1$, $n(\gl)=5$. Furthermore
\begin{align*}u^+_{\gl}&=(x_1-q_2)(x_2-q_2)(x_3-q_2)(x_4-q_2)(x_1-q_3)
(x_2-q_3)(x_3-q_3)(x_4-q_3)(x_5-q_3)(x_6-q_3)(x_7-q_3),\\
u^-_{\bar{\gl}}&=-(x_1-q_1)(x_2-q_1)(x_3-q_1)(x_4-q_1)(x_1-q_2),\\
w_\lambda&=\Big(\begin{array}{*4r|*3r|*2r} 1& 2& 3& 4& 5& 6& 7& 8\\
                              8& 5& 7& 6& 1& 3& 4& 2
              \end{array}\Big)=(1,8,2,5)(3,7,4,6),
w_{\bar{\gl}}\,=\Big(\begin{array}{*4r|*3r|*2r} 1& 2& 3& 4& 5& 6& 7& 8\\
                                   5&8&6& 7& 2& 4& 3& 1
	           \end{array}\Big)=(1,5,2,8)(3,6,4,7)=w_{\gl}^{-1},
              \end{align*}
and $w_{\gl}\cdot u^-_{\bar{\gl}}=-(x_8-q_1)(x_5-q_1)(x_7-q_1)(x_6-q_1)(x_8-q_2)$. Thus
$L(u^+_{\gl}(w_{\gl}\!\cdot\!u^-_{\bar{\gl}}))=-
 x_1^2x_2^2\cdots x_8^2$. similarly, it follows by direct computation that
$L((w_{\bar{\gl}}\!\cdot\!u^+_{\gl})u^-_{\bar{\gl}})=-
 x_1^2x_2^2\cdots x_8^2$.
\end{example}

The following fact is the key to our computation of $\tau(z_\gl w_{\bar{\gl}})$.

\begin{lemma}\label{Lem key}Let $\gl$ be an $m$-multipartition of $n$ and let $l=(m-1)n$. Then the polynomials
$u^+_{\gl}(w_{\gl}\!\cdot\!u^-_{\bar{\gl}})$
 and $(w_{\bar{\gl}}\!\cdot\!u^+_{\gl})u^-_{\bar{\gl}}$ are of degree $(m-1)n$, moreover,
 $$\gr_{l}u^+_{\gl}(w_{\gl}\!\cdot\!u^-_{\bar{\gl}})= \gr_{l}(w_{\bar{\gl}}\!\cdot\!u^+_{\gl})u^-_{\bar{\gl}}=(-1)^{n(\gl)}x_1^{m-1}x_2^{m-1}\cdots x_n^{m-1}.$$
 \end{lemma}

\begin{proof}It suffices to show that $\gr_{l}(w_{\bar{\gl}}\!\cdot\!u^+_{\gl})u^-_{\bar{\gl}}=(-1)^{n(\gl)}x_1^{m-1}
x_2^{m-1}\cdots x_n^{m-1}$. Indeed, if this is done, then the first part follows immediately from that the degrees of
 $u^+_{\gl}(w_{\gl}\!\cdot \!u^-_{\bar{\gl}})$ and $(w_{\bar{\gl}}\!\cdot\!u^+_{\gl})u^-_{\bar{\gl}}$ are at most
 $(m-1)n$; on the other hand, note that for any $w\in S_n$, $\gr_{l}w\!\cdot\!(w_{\gl}\!\cdot\!u^+_{\gl})u^-_{\bar{\gl}}=
(-1)^{n(\gl)}x_1^{m-1}x_2^{m-1}\cdots x_n^{m-1}$.  In particular, let $w=w_{\gl}$, then
$\gr_{l}w_{\gl}\!\cdot\!(w_{\bar{\gl}}\!\cdot\!u^+_{\gl})u^-_{\bar{\gl}}=
\gr_{l}(w_{\gl}w_{\bar{\gl}}\!\cdot\!u^+_{\gl})(w_{\gl}\!\cdot\!u^-_{\bar{\gl}})=
\gr_{l}u^+_{\gl}(w_{\gl}\!\cdot\!u^-_{\bar{\gl}})=
(-1)^{n(\gl)}x_1^{m-1}x_2^{m-1}\cdots x_n^{m-1}$ since $w_{\bar{\gl}}w_{\gl}=1$.

We proceed by induction on $n$. Then, by definitions, $L(u^+_{\gl})=\prod_{i=1}^n x_{i}^{m-c^{\gl}(i)}$ and
 $L(u^-_{\bar{\gl}})=\prod_{i=1}^n x_{i}^{c_{\gl}(i)-1}$.  Now, if  there are some $i$
with $1\le i\le m$ such that $|\gl^i|=0$, for simplicity, we assume that $\gl^1=\emptyset$ and
 $\gl^2\neq \emptyset\neq \gl^m$, then
 \begin{align*}c^{\gl}(i)=\left\{
      \begin{array}{ll}
        2, & \hbox{if } 1\le i\le |\gl^2|;\\
        m, & \hbox{if } n-|\gl^{m}|< i\le n,
      \end{array}
    \right.\qquad\text{ and }\qquad c_{\gl}(i)=\left\{
      \begin{array}{ll}
        2, & \hbox{if } n-|\gl^2|<i\le n;\\
        m, & \hbox{if } 1\le i\le |\gl^m|.
      \end{array}
    \right.\end{align*}
Therefore
 \begin{align*}L(u^+_{\gl})=\displaystyle\prod_{i=1}^{|\gl^2|}x_i^{m-2}\displaystyle\prod_{i>|\gl_2|}^{n-|\gl^m|}
 x_{i}^{m-c_{\gl}(i)} \quad\text{ and }\quad L(u^-_{\bar{\gl}})=(-1)^{n(\gl)}\displaystyle\prod_{i=1}^{|\gl^m|}x_i^{m-1}
  \displaystyle\prod_{i>|\gl^m|}^{n-|\gl_2|} x_{i}^{c_{\gl}(i)}\displaystyle\prod_{i>n-|\gl^2|}^nx_{i}.\end{align*}
On the other hand,
\begin{align*}
w_{\bar{\lambda}}&=\Big(\begin{array}{*4c|*2c|*4c} 1& 2& \cdots& |\gl^2|& \cdots& \cdots& n-|\gl^m|+1& \cdots& n\\
                                   i_{1}&i_{2}&\cdots& i_{|\gl^2|}& \cdots& \cdots& j_1& \cdots& j_{|\gl^m|}
	           \end{array}\Big)
\end{align*}
where $\{i_1=n-|\gl^2|+1, i_2, \cdots, i_{|\gl^2|}\}=\{n-|\gl^2|+1, \cdots, n\}$ and
$\{j_1=1, j_2, \cdots, j_{|\gl^m|}\}=\{1, \cdots, |\gl^m|\}$.

\noindent As a consequence,   $$L((w_{\bar{\gl}}\!\cdot\!u^+_{\gl})u^-_{\bar{\gl}})=
(-1)^{n(\gl)}\displaystyle\prod_{i=1}^{|\gl^m|}x_i^{m-1}\prod_{i>|\gl_2|}^{n-|\gl^m|}
 x_{w_{\gl}(i)}^{m-c_{\gl}(i)}\prod_{i>|\gl^m|}^{n-|\gl_2|} x_{i}^{c_{\gl}(i)} \prod_{i>n-|\gl^2|}^nx_{i}^{m-1}.$$
Now, by induction, $L((w_{\bar{\gl}}\!\cdot\!u^+_{\gl})u^-_{\bar{\gl}})=(-1)^{n(\gl)}x_1^{m-1}x_2^{m-1}\cdots x_n^{m-1}$.

Now assume that $\gl=(\gl^1; \dots; \gl^m)$ with $|\gl^i|>0$. Then \begin{align*}L(u^+_{\gl})=\prod_{i=1}^{|\gl^1|}x_i^{m-1}
\prod_{i>|\gl_1|}^{n-|\gl^m|}
 x_{i}^{m-c_{\gl}(i)} \quad\text{ and }\quad L(u^-_{\bar{\gl}})=(-1)^{n(\gl)}\prod_{i=1}^{|\gl^m|}x_i^{m-1}
  \prod_{i>|\gl^m|}^{n-|\gl_1|} x_{i}^{c_{\gl}(i)}.\end{align*}
Using the same arguments, we yield that $L((w_{\bar{\gl}}\!\cdot\!u^+_{\gl})u^-_{\bar{\gl}})=(-1)^{n(\gl)}x_1^{m-1}x_2^{m-1}\cdots x_n^{m-1}$.
\end{proof}

Now we can obtain the main result of this section.

\begin{theorem}\label{Them tau(Zw)}
Let $\gl$ be an $m$-multipartition of $n$ and $n(\gl)\!\!=\!\!\sum_{i=1}^m(i\!-\!1)|\gl^i|$.
Then $\tau(z_\gl w_{\bar{\gl}})\!\!=\!\!(\!-1\!)^{n(\gl)}$.
\end{theorem}
\begin{proof}
First, note that for any $w\in S_n$ and any polynomial $f(x_1, \dots, x_n)$ (in $x_1, \dots, x_n$),
$f$ and $w\!\cdot\!f$ have the same degree, furthermore, the leading terms of $w\!\cdot\! f$ are obtained from those of
$f$ via the $w$-action. Let $l=(m-1)n$.  Therefore
\begin{align*}\tau(z_\gl w_{\bar{\gl}})&=\tau(x_{\gl}u^+_{\gl}w_{\gl}u^-_{\bar{\gl}}y_{\bar{\gl}}w_{\bar{\gl}})
 &(\text{Definition }\ref{Def Z_lambda})\\
&=\tau(u^+_{\gl}w_{\gl}u^-_{\bar{\gl}}y_{\bar{\gl}}w_{\bar{\gl}}x_{\gl})&(\ref{tau properties}\text{(i)})\\
&=\hat{\tau}\left(\gr_{l}(u^+_{\gl}w_{\gl}u^-_{\bar{\gl}}y_{\bar{\gl}}w_{\bar{\gl}}x_{\gl})\right)&(\text{Theorem }\ref{Thm-def tau})\\
&=\hat{\tau}\left(\gr_{l}\left(u^+_{\gl}(w_{\gl}\!\cdot\!u^-_{\bar{\gl}})\right)w_{\gl}y_{\bar{\gl}}w_{\bar{\gl}}x_{\gl}\right)&(\text{Lemma }\ref{Lem graded iso})\\
&=(-1)^{n(\gl)}\sum_{u\in S_{\bar{\gl}}, v\in S_{\gl}}(-1)^{\ell(u)}\tau\left(x_1^{m-1}\cdots x_n^{m-1}w_{\gl}uw_{\bar{\gl}}v\right)&(\text{Lemma }\ref{Lem key})\\
&=(-1)^{n(\gl)}\sum_{u\in S_{\bar{\gl}}, v\in S_{\gl}}(-1)^{\ell(u)}\tau\left(x_1^{m-1}\cdots x_n^{m-1}w_{\gl}uw_{\gl}^{-1}v\right)&( w_{\gl}w_{\bar{\gl}}=1)\\
&=(-1)^{n(\gl)}\sum_{\substack{(u,v)\in S_{\bar{\gl}}\times S_{\gl}\\ w_{\gl}uw_{\gl}^{-1}=v^{-1}}}(-1)^{\ell(u)}\tau\left(x_1^{m-1}
\cdots x_n^{m-1}w_{\gl}uw_{\gl}^{-1}v\right)&( \ref{tau properties}\text{(ii)})\\
&=(-1)^{n(\gl)}\tau(x_1^{m-1}
\cdots x_n^{m-1})&(S_{\gl}\cap w_{\gl}S_{\bar{\gl}}w_{\gl}^{-1}=\{1\})\\
&=(-1)^{n(\gl)}.
\end{align*}
 \end{proof}

\begin{remark}Note that we do not use the Assumption~\ref{ppoly} in this section. The Theorem shows that for all parameters $q_1, \dots, q_m$, $\tau$ is a non-degenerate trace form on the degenerate cyclotomic Hecke algebra $\sh_{m,n}(Q)$, that is, $\sh_{m,n}(Q)$ is a symmetric algebra for all
parameters $q_1, \dots, q_m$, which gives a differential proof of the non-degeneration of the trace form $\tau$. \end{remark}
\section{The Schur elements}\label{Sec: Schur elements}
In this section we compute the Schur elements of the degenerate cyclotomic Hecke algebra $\sh$.
Assume that Assumption~\ref{ppoly} holds.  In this semisimple case
$\{S^\lambda|\lambda \in \mpn\}$ is a complete set of pairwise
non-isomorphic irreducible $\sh$-modules.  Let $\chi^\lambda$ be the
character of $S^\lambda$. Following Geck's results on symmetrizing form (see \cite[Theorem 7.2.6]{Geck}),
we obtain the following definition for the Schur elements of $\sh$ associated to the irreducible representations
of $\sh$.

\begin{definition}
Suppose that $R$ is a field and that $P_\sh(Q)\ne0$. The {\it Schur
elements} of $\sh$ are the elements $s_\lambda(Q)\in R$ such that
$$\tau=\sum_{\lambda\in\mpn} \frac1{s_\lambda(Q)}\chi^\lambda.$$
The rational functions $\frac1{s_\lambda(Q)}$ are also
called the \textit{weights} of $\sh$.
\end{definition}

Recall that $F_{\ft^{\gl}}$ is a primitive idempotent in $\sh$ such that
$S^\lambda\cong \sh F_{\ft^{\gl}}$ for each $m$-multipartiton of $n$
and $\{F_{\ft^{\gl}}\mid\gl\in\mpn\}$ is the set
of primitive idempotents  in $\sh$  (see Proposition~\ref{Prop F_t^lambda Z_lambda}). By applying a well-known
fact about symmetric algebras (see \cite[Proposition 9.17]{Curtis-Reiner}).
We can compute the Schur elements of $\sh$ by following the
Lemma.

\begin{lemma} Assume that $R$ is a field and that $\sh$ is semisimple.
Let $\lambda$ be an $m$-multipartition of $n$. Then $s_\lambda(Q)=\displaystyle\frac
1{\tau(F_{\ft^{\gl}})}$.
\label{Lemma Schur}
\end{lemma}

Before we begin our computation, we consider first an example.
\begin{example}
Fix $i$ with $1\le i\le m$ and let
$\eta_i=(\eta_i^{1};\dots;\eta_i^{m})\in\mpn$ with $\eta_i^{j}=\delta_{ij}n$. We will compute
 the Schur elements $s_{\eta_i}(Q)$.  Let
$x_{\eta_i}=\sum_{w\in S_n}w$ and $u_{\eta_i}=\prod_{j\ne
i}\prod_{k=1}^n(x_k-q_j)$, and
set $e_{\eta_i}=u_{\eta_i} x_{\eta_i}=x_{\eta_i}u_{\eta_i}$
(cf~(\ref{std basis})).  It follows from Lemma~\ref{Lem si xj} that $u_{\eta_i}$
is central in~$\sh$.  Further,  the relations imply that
$x_1u_{\eta_i}=q_iu_{\eta_i}$ and $w
x_{\eta_i}=x_{\eta_t}$ for $w\in S_n$; it follows that
$x_k e_{\eta_t}=(k-1+q_i)e_{\eta_i}$ for $k=1,\dots,n$. Thus
the module $\sh m_{\eta_i}$ is one dimensional and, in particular,
irreducible; in fact by Theorem~\ref{x_k action}(ii) and \S2.8, $S^{\eta_i}\cong \sh e_{\eta_i}=Re_{\eta_i}$.
 Moreover, by what we have said
$$e_{\eta_i}^2 =n!\prod_{j\ne i}
         \prod_{k=1}^n(k-1+q_i-q_j)\cdot e_{\eta_e}.$$
So $e_{\eta_i}$ is a
scalar multiple of the primitive idempotent which generates $S^{\eta_i}$.
Hence, by the Lemma,
$$\begin{aligned}
s_{\eta_i}(Q)
   &=n!\prod_{j\ne i}\prod_{k=0}^{n-1}(k+q_i-q_j).\\
   \end{aligned}$$
Similar arguments give the Schur elements for the multipartition
which is conjugate to $\eta_i$; alternatively, they are given by
Corollary~\ref{symmetry} and the calculation above.
\label{Example eta_t}\end{example}

\begin{remark} There is an action of $S_m$ on the set of $m$-multipartitions of $n$
(by permuting components) and also on the rational functions in
$q_1,\dots,q_m$ (by permuting parameters). When $\sh$ is semisimple the
Specht modules are determined up to isomorphism by the action of
$x_1,\dots,x_n$; as the relation $\prod_{i=1}^m(x_1-q_i)=0$ is
invariant under the $S_m$-action it follows that
$s_{v\cdot\lambda}(Q)=v\cdot s_\lambda(Q)$ for all $m$-multipartitions
$\lambda$ and all $v\in S_m$; this is also clear from
Theorem~\ref{Them idempotents}(i). In the case where $\lambda=\eta_i$ this symmetry
is evident in the formulae above.\end{remark}

By Lemma~\ref{Lemma Schur} and Theorem~\ref{Them idempotents} the Schur elements
are given by $s_\lambda(Q)=\tau(F_{\ft^\gl})^{-1}$. Proposition~\ref{F'} implies that the
Schur elements have the following ``palindromy'' property.

\begin{corollary}
Suppose that $\lambda$ is an $m$-multipartition of $n$. Then
$s_{\bar{\gl}}(Q)=\overline{s_\lambda(Q)}$.
\label{symmetry}
\end{corollary}

The following fact givse the formula for the Schur elements of $s_{\gl}(Q)$ of $\sh$.

\begin{proposition}\label{Prop s_lamda(Q)}
Assume that Assumption~\ref{ppoly} holds. Let $\gl$ be an $m$-multipartition of $n$. Then $s_{\gl}(Q)=(-1)^{n(\gl)}
\gma_{\ft_{\gl}}\overline{\gma_{\ft^{\bar{\gl}}}}$.
\end{proposition}
\begin{proof}The  Proposition follows directly by Proposition~\ref{Prop F_t^lambda Z_lambda}, Lemma~\ref{Lemma Schur}, and
Theorem~\ref{Them tau(Zw)}.
\end{proof}

A closed formula for  $\overline{\gma_{\ft^{\bar{\gl}}}}$ can be given by Theorem~\ref{gamma properties}(ii.a).
Now we determine $\gma_{\ft_{\gl}}$ which enables us to give an explicit formula for $s_{\gl}(Q)$. Before to do this, we need some notation.

\begin{Point}{}*\,Recall that the $(i,j)$-th {\it hook} in the diagram $[\lambda^{s}]$ is the
collection of nodes to the right of and below the node $(i,j,s)$,
including the node $(i,j,s)$ itself. The $(i,j)$-th {\it hook length}
$h^{\lambda^{s}}_{ij}=\lambda^{s}_i+\overline{\lambda^{s}_j}-i-j+1$ is the
number of nodes in the $ij$th hook and the {\it leg length},
$\ell^{\lambda^{s}}_{ij}=\overline{\lambda^{s}_j}-j+1$, is the number of nodes
in the ``leg'' of this hook. Observe that if $(a,b,c)$ and $(i,j,c)$
are two removable nodes in $[\lambda^{(c)}]$ with $a\le i$
and $j\le b$ then $h^{\lambda^{s}}_{aj}=b-a-j+i+1$.
\end{Point}

\begin{lemma}
Suppose that $\lambda$ is an $m$-multipartition of $n$. Then
$$\gamma_{\ft_{\gl}}=\prod_{(i,j,s)\in[\lambda]}\!\!
    \frac{h^{\lambda^{s}}_{ij}}{\ell^{\lambda^{s}}_{ij}}
    \prod_{t=s+1}^m\!\! \biggl((j-i+q_s- \gl_1^{t}-q_t)
              \prod_{k=1}^{\lambda^{t}_1}
	            \frac{j-i+q_s-k+1+\overline{\lambda^{t}_k}-q_t}
	                 {j-i+q_s-k+\overline{\lambda^{t}_k}-q_t}\biggr).
$$
\label{Lem gamma_t_lambda}
\end{lemma}
\begin{proof}
We proceed by induction on $n$.
If $n=0$, both sides are $1$
and there is nothing to prove (by convention, empty products are $1$).
Suppose that $n>0$. Let $\mu=\shape(\ft_{\gl}\!\downarrow\!n-1)$. Then
$\mu$ is an $m$-multipartition of $n-1$.
Applying Theorem~\ref{gamma properties}(i),
$$\frac{\gamma_{\ft_{\gl}}}{\gamma_{\ft_\mu}}
  =\frac{\prod_{x\in \mathscr A_{\ft_{\gl}}(n)}\(\res_{\ft_{\gl}}(n)-\res(x)\)}%
     {\prod_{y\in \mathscr R_{\ft_{\gl}}(n)}\(\res_{\ft_{\gl}}(n)-\res(y)\)}.$$
Assume that the integer $n$ appears in node $(a,b,c)$ of
${\ft_{\gl}}$.  First consider the contribution that the
addable and removable nodes in $[\lambda^{c}]$ make to
$\gamma_{\ft_{\gl}}$. By definitions, these nodes occur in pairs $(x,y)$
where $y\succ(a,b,c)$ is a removable node in row $i$ and $x\succ(a,b,c)$ is an addable
node in row $i+1$ for some $i\ge a$. If $x$ is in column $d$ of
$[\lambda^{c}]$ and $y$ is in column $d'$ then $d\le d'<b$ and
\begin{align*}
\frac{\res_{\ft_{\gl}}(n)-\res_{\ft_{\gl}}(x)}{\res_{\ft_{\gl}}(n)-\res_{\ft_{\gl}}(y)}
  &=\frac{b-a+q_c- d+(i+1)-q_c}{b-a+q_c- d'+i-q_c}
   =\prod_{j=d}^{d'}\frac{b-a- j+(i+1)}{b-a- j+i}\\
  &=\prod_{j=d}^{d'}\frac{j-i-1 b-a-j+i+1 -1}%
                         {j-i b-a-j+i-1)}
   =\prod_{j=d}^{d'}\frac{ h^{\lambda^{c}}_{aj}-1}
                       {h^{\lambda^{c}}_{aj}-1 }\\
  &=\prod_{j=d}^{d'}\frac{h^{\lambda^{c}}_{aj}}
                       {h^{\lambda^{(c)}}_{aj}-1}.
\end{align*}
Therefore,
\begin{align*}
\frac{\prod_{(i,j,c)\in \mathscr A_{\ft_{\gl}}(n)}
                \(\res_{\ft_{\gl}}(n)-\res_{\ft_{\gl}}(i,j,c)\)}%
     {\prod_{(i,j,c)\in \mathscr R_{\ft_{\gl}}(n)}
                \(\res_{\ft_{\gl}}(n)-\res_{\ft_{\gl}}(i,j,c)\)}
     =\prod_{j=1}^{b-1}\frac{h^{\lambda^{c}}_{aj}}
                       {h^{\lambda^{c}}_{aj}\!-1}
    =\prod_{j=1}^{b-1}
       \frac{h^{\lambda^{c}}_{aj}}{h^{\mu^{c}}_{aj}}.
\end{align*}
\indent Now $\gamma_{\ft_\mu}$ is known by induction and it contains as a
factor the left hand term in the product below. Further,
$\ell^{\mu^{s}}_{ij}=\ell^{\lambda^{s}}_{ij}$ and
$h^{\mu^{s}}_{ij}=h^{\lambda^{s}}_{ij}$ if $(i,s)\ne(a,c)$,
and $\ell^{\mu^{c}}_{aj}=\ell^{\lambda^{c}}_{aj}$ for
$1\le j<b$,~so
\begin{align*}
\bigg(\prod_{(i,j,s)\in[\mu]}%
         \frac{h^{\mu^{s}}_{ij}}{\ell^{\mu^{s}}_{ij}}\bigg)%
	 \bigg(\prod_{j=1}^{b-1}
	 \frac{h^{\lambda^{c}}_{aj}}{h^{\mu^{c}}_{aj}}\bigg)
    &=\bigg(\prod_{\substack{(i,j,s)\in[\lambda]\\ (i,s)\ne(a,c)}}%
      \frac{h^{\lambda^{s}}_{ij}}{\ell^{\lambda^{s}}_{ij}}
     \bigg)\bigg(%
     \prod_{j=1}^{b-1}
       \frac{h^{\lambda^{c}}_{aj}}{ \ell^{\lambda^{c}}_{aj}}
     \bigg)\\
   &=\prod_{(i,j,s)\in[\lambda]}%
      \frac{h^{\lambda^{s}}_{ij}}{\ell^{\lambda^{s}}_{ij}},
\end{align*}
since $h^{\lambda^{c}}_{ab}=1=\ell^{\lambda^{c}}_{ab}$.
This accounts for the left hand factor in the expression
for~$\gamma_{\ft_{\gl}}$ given in the statement of the Lemma.

Finally, consider the nodes in
$\mathscr A_{\ft_{\gl}}(n)$ and $\mathscr R_{\ft_{\gl}}(n)$ which are in
component $t$ for some $t>c$ (there are no such nodes for $t<c$).
Again, almost all of the addable and removable nodes in component $t$
occur in pairs placed in consecutive rows; however, this time there is
also an additional addable node at the end of the first row
of $\lambda^{t}$. As above, it is easier to insert extra factors
which cancel out and so take a product over all of the columns
of $\lambda^{(t)}$. An argument similar to that above shows that the
nodes in $\mathscr A_{\ft_{\gl}}(n)$ and $\mathscr R_{\ft_{\gl}}(n)$ which do
not belong to component~$c$ contribute the factor
$$\prod_{t=c+1}^r (b-a+q_c-\lambda^{t}_1 -q_t)
  \prod_{k=1}^{\lambda^{t}_1}
      \frac{(b-a+q_c- k+1+\overline{\lambda^{t}_k}-q_t)}
           {(b-a+q_c- k+\overline{\lambda^{t}_k}-q_t)}$$
to $\gamma_{\ft_{\gl}}$. Using induction to combine the formulae above
proves the Lemma.
\end{proof}

Now we obtain the explicit formulae for the Schur elements of the degenerate Hecke algebras $\sh$.

\begin{theorem}\label{Them Schur elements}Let $\gl$ be an $m$-multipartition of $n$. Then
\begin{align*}s_{\gl}(Q)=\prod_{(i,j,s)\in[\lambda]}\!\!%
    h^{\lambda^{s}}_{ij}\prod_{1\le s<t\le m}X^{\gl}_{st},\end{align*}
where, for $1\le s<t\le m$,
\begin{align*}X^{\gl}_{st}=\prod_{(k,l)\in[\gl^t]}(l-k+q_t-q_s)\prod_{(i,j)\in [\gl^s]}\biggl((j-i+q_s-\gl_1^t-q_t)
\prod_{r=1}^{\gl_1^t}\frac{j-i+q_s-r+1+\bar{\gl_1^t}-q_t}{j-i+q_s-r-q_t}\biggr).\end{align*}
\end{theorem}
\begin{proof}  By applying Theorem~\ref{gamma properties}(ii.a), we obtain that
$$(\dag)\quad\quad\gamma_{\ft^{\bar{\gl}}}=\bar{\gl}!
  \prod_{1\le s<t\le m}\prod_{(i,j,s)\in[\bar{\lambda}]}(j-i+q_t-q_s).$$
Observe that $\bar{\gl}!=\prod_{(i,j,s)\in[\lambda]}\ell^{\gl^s}_{ij}$. Further $(i,j)\in[\gl^s]$ if and only if
$(j, i)\in[\overline{\gl^{m-s+1}}]$. Therefore applying the $\bar{\,\,}$ operation on the equality $(\dag)$, and swapping
the roles $s$ and $t$ in the right-hand factor,
\begin{align*}
\overline{\gamma_{\ft^{\bar{\gl}}}}&=\bar{\gl}!
  \prod_{1\le s<t\le m}\prod_{(i,j)\in[\overline{\lambda^s}]}(j-i+\overline{q_t}-\overline{q_s})\\
&=\bar{\gl}!\prod_{1\le s<t\le m}\prod_{(j,i)\in[\lambda^{m-s+1}]}(j-i+q_{m-t+1}-q_{m-s+1})\\
&=(-1)^{n(\gl)}\bar{\gl}!\prod_{1\le s<t\le m}\prod_{(j,i)\in[\lambda^{m-s+1}]}(i-j+q_{m-s+1}-q_{m-t+1})\\
&=(-1)^{n(\gl)}\prod_{(i,j,s)\in[\gl]}\ell^{\gl^s}_{ij}\prod_{1\le s<t\le m}\prod_{(i,j)\in[\lambda^{t}]}(j-i+q_{t}-q_{s}).
\end{align*}
Now using Proposition~\ref{Prop s_lamda(Q)} and Lemma~\ref{Lem gamma_t_lambda},
\begin{align*}
s_{\gl}(Q)&=\prod_{s=1}^m\prod_{(i,j)\in[\lambda^s]}\!\!%
    h_{ij}^{\gl^s}\prod_{t=s+1}^m\prod_{(k,l)\in[\gl^t]}(l-k+q_t-q_s)\\
    &\qquad\quad\qquad\quad\qquad\quad\prod_{(i,j)\in[\gl^s]}\biggl((j-i+q_s-q_t)
    \prod_{r=1}^{\gl_1^t}\frac{j-i+q_s-r+1+\overline{\gl^t_r}-q_t}{j-i+q_s-r+\overline{\gl^t_r}-q_t}\biggr)\\
    &=\prod_{s=1}^m\prod_{(i,j)\in[\lambda^s]}\!\!%
    h_{ij}^{\gl^s}\prod_{t=s+1}^mX^{\gl}_{st}.
\end{align*}
\end{proof}

\begin{example}It is straightforward to check that the Theorem gives the same rational functions for the Schur
elements $s_{\eta_t}(Q)$ as were obtained in Example~\ref{Example eta_t}.
\end{example}

\begin{remark}In a subsequent paper, we will give a symmetric and cancellation-free formula for the Schur elements for the degenerate cyclotomic Hecke algebras and some applications of our formula. \end{remark}

\end{document}